\documentclass[reqno,12pt]{amsart} 
\usepackage{amsmath,amssymb,amsthm,enumerate,mathrsfs,colonequals,bm}
\usepackage{fullpage,url,tikz,xspace,setspace, mleftright}
\usepackage{microtype}
\usepackage{calc}
\usepackage{graphicx}
\usepackage[
unicode=true
]{hyperref}
\usepackage[alphabetic,lite,nobysame]{amsrefs} % for bibliography

\usepackage{color,tikz}
\usetikzlibrary{decorations.markings}
\newcommand{\bruce}[1]{\textsf{\color{blue} $\spadesuit$ Bruce: #1}}

\let\oldtocsection=\tocsection
\let\oldtocsubsection=\tocsubsection
\let\oldtocsubsubsection=\tocsubsubsection
\renewcommand{\tocsection}[2]{\hspace{0em}\oldtocsection{#1}{#2}}
\renewcommand{\tocsubsection}[2]{\hspace{1em}\oldtocsubsection{#1}{#2}}
\renewcommand{\tocsubsubsection}[2]{\hspace{2em}\oldtocsubsubsection{#1}{#2}}

\newcommand{\Bmu}{\mbox{$\raisebox{-0.59ex}
  {$l$}\hspace{-0.18em}\mu\hspace{-0.88em}\raisebox{-0.98ex}{\scalebox{2}
  {$\color{white}.$}}\hspace{-0.416em}\raisebox{+0.88ex}
  {$\color{white}.$}\hspace{0.46em}$}{}}

\newtheorem{theorem}{Theorem}
\newtheorem{proposition}[theorem]{Proposition}
\newtheorem{lemma}[theorem]{Lemma}
\newtheorem{corollary}[theorem]{Corollary}
\newtheorem{definition}[theorem]{Definition}
\newtheorem{remark1}[theorem]{Remark}

\newcommand{\Z}{\mathbb{Z}}
\newcommand{\Q}{\mathbb{Q}}

\newcommand{\HH}{\mathbb{H}}
\newcommand{\NN}{\mathbb{N}}
\newcommand{\F}{\mathbb{F}}

\newcommand{\bFp}{\overline{\F}_p}
\newcommand{\R}{\mathbb{R}}
\newcommand{\RP}{\mathbb{RP}}
\renewcommand{\H}{\mathbb{H}}

\newcommand{\A}{\mathscr{A}}
\newcommand{\cA}{\mathcal{A}}
\newcommand{\X}{\mathscr{X}}

\newcommand{\bs}{\backslash}

\newcommand{\cO}{\mathcal{O}}
\newcommand{\OO}{\cO}
\newcommand{\cB}{\mathcal{B}}

\newcommand{\cK}{K}

\newcommand{\bfr}{\mathbf{r}}
\newcommand{\obr}{\overline{\bfr}}
\newcommand{\hbr}{\hat{\bfr}}

\newcommand{\gr}{\mathit{gr}}

\newcommand{\Disc}{\rm Disc}
\newcommand{\co}{\mathit{co}}

\newcommand{\Gr}{\mathit{Gr}}
\newcommand{\wgr}{\widetilde{\mathit{gr}}}
\newcommand{\wco}{\widetilde{\mathit{co}}}

\newcommand{\brc}{\mathit{brc}}
\newcommand{\wbrc}{\widetilde{\mathit{brc}}}
\newcommand{\sS}{{\sf SP}}
\newcommand{\Sk}{{\sf Sk}}
\newcommand{\sSgpz}{\sS_{\!g}(p)_0}

\newcommand{\osSgpl}{\overline{\sS}_{\!g}(\ell,p)}

\newcommand{\sSgpr}{\sS_{\!g}(\ell, p)_r}
\newcommand{\sSgps}{\sS_{\!g}(\ell, p)_s}
\newcommand{\sSgpgr}{\sS_{\!g}(\ell, p)_{\hat{r}}}

\newcommand{\sSgpl}{\sS_{\!g}(\ell, p)}

\newcommand{\zm}{\begin{bmatrix}\cdot 0\cdot\end{bmatrix}}

\DeclareMathOperator\End{End}
\DeclareMathOperator\SL{SL}
\DeclareMathOperator\SLT{\SL_{2}}
\DeclareMathOperator\Id{Id}
\DeclareMathOperator\Sp{Sp}
\DeclareMathOperator\GSp{GSp}
\DeclareMathOperator\PGSp{PGSp}
\DeclareMathOperator\rU{U}
\DeclareMathOperator\GU{GU}

\newcommand\Ram{{\sf Ram\,}}

\newcommand\Et{{\sf Et\,}}
\DeclareMathOperator\Ta{Ta}
\DeclareMathOperator\Ver{Ver}
\DeclareMathOperator\Ed{Ed}
\DeclareMathOperator\Pic{Pic}

\DeclareMathOperator\Iso{Iso}
\DeclareMathOperator\iso{iso}
\DeclareMathOperator\Mat{Mat}
\DeclareMathOperator\Adj{Adj}
\DeclareMathOperator\Aut{Aut}
\DeclareMathOperator\Ad{Ad}
\DeclareMathOperator\Adw{\Ad_w}

\DeclareMathOperator\charr{char}
\DeclareMathOperator\disc{disc}
\DeclareMathOperator\w{w}
\DeclareMathOperator\Grr{Gr}

\DeclareMathOperator\rdeg{rdeg}

\DeclareMathOperator\Orb{Orb}
\DeclareMathOperator\ty{\mathit t}
\DeclareMathOperator\Fa{Fa}
\DeclareMathOperator\Adjj{Adj}

\begin{document}
\title{Isogeny complexes of
superspecial abelian varieties}
\subjclass[2010]{Primary 14K02; Secondary 11G10, 14G15}
\keywords{superspecial, abelian varieties, isogeny graphs, Brandt matrices,
quaternionic unitary group}

\author{Bruce~W.~Jordan}
\address{Department of Mathematics, Baruch College, The City University
of New York, One Bernard Baruch Way, New York, NY 10010-5526, USA}
\email{bruce.jordan@baruch.cuny.edu}

\author{Yevgeny~Zaytman}
\address{Independent Mathematician,
Newton, MA 02465, USA}
\email{gzaytman@alum.mit.edu}

\begin{abstract}
We consider the structures formed by isogenies of 
abelian varieties with polarizations that are not necessarily principal, 
specifically with the
$[\ell]$-polarizations we have previously defined.  Our primary interest is
in superspecial abelian varieties, where the isogenies are related 
to quaternionic hermitian forms.  We first consider isogeny graphs.
We show that these $[\ell]$-isogeny graphs are a generalized
Brandt graph and construct
them entirely in terms of definite quaternion algebras.  We prove 
that they are connected and give examples to 
show that the regular graphs obtained are sometimes Ramanujan and 
sometimes not. Isogenies of $[\ell]$-polarized abelian varieties can be closed
under composition, with the consequence that such isogenies naturally
form semi-simplicial complexes as introduced by Eilenberg and Zilber in 1950
(later also called 
$\Delta$-complexes)---the higher-dimensional analogues of multigraphs.
  We show that these isogeny 
complexes can be constructed from the arithmetic of hermitian forms over 
definite quaternion algebras and that they are quotients of the Bruhat-Tits
building of the symplectic group by the action of a quaternionic unitary group.
Working with quaternions these isogeny graphs and complexes are amenable to machine computation and we include many examples, concluding with a detailed examination of the $[2]$-isogeny complexes of superspecial abelian surfaces in characteristic $7$.

\end{abstract}

\vspace*{.3in}

\maketitle

\tableofcontents

\section{Introduction}

Fix primes $p$ and $\ell$ with $\ell\neq p$.
Isogenies of degree $\ell$ (or $\ell$-isogenies) of supersingular elliptic curves in characteristic
$p$  are 
naturally organized
into isogeny graphs, which can in turn be described
using definite rational quaternion algebras.  The only subtlety
is that there are variations depending on how polarizations and 
isogenies are identified, leading to {\em three} different 
isogeny graphs in Jordan-Zaytman \cite{jz}: the 
big isogeny graph $\Gr_{1}(\ell,p)$, the little isogeny graph $\gr_{1}(\ell,p)$, and the 
enhanced isogeny graph $\wgr_{1}(\ell,p)$.  The literature seems 
to have only one 
isogeny graph; this ubiquitous graph is the big isogeny graph for us. 

Distinguishing between these three 
makes many results clearer and more precise.  For example,
the little and enhanced isogeny graphs are uniformized by the 
Bruhat-Tits building $\Delta=\Delta_{\ell}$ of $\SLT(\Q_{\ell})$, whereas
the big isogeny graph is not \cite[Sect.~9.1]{jz}.  
The big isogeny graph is a regular graph,
so it is natural to ask if it is Ramanujan, whereas the little and 
enhanced isogeny graphs are not.  And it is the little and enhanced isogeny
graphs which arise from the bad reduction of Shimura curves and not
the familiar big isogeny graph \cite[Sect.~9.2]{jz}.

Fix a supersingular elliptic curve $E$ over $\bFp$.  Then $\OO=\OO_E=\End(E)$
is a  maximal order in the definite rational quaternion algebra $\HH_p$ ramified at $p$. A 
superspecial abelian variety $A$ over $\bFp$ of dimension $g$ 
is isomorphic to $E^g$; this is equivalent to requiring that $A$ 
is isomorphic to a product of supersingular elliptic curves.
A polarization $\lambda$ of $A$ is a symmetric isogeny $\lambda:A\rightarrow  
\hat{A}=\Pic^0(A)$ subject to the positivity condition involving
the Poincar\'{e} line bundle given in Section \ref{type}.
 The degree $\deg(\lambda)$ of a polarization $\lambda$ is the 
degree of $\lambda$ as an isogeny. It is always a square
and the reduced degree $\rdeg(\lambda)$ of $\lambda$ is 
$\rdeg(\lambda)=\sqrt{\deg(\lambda)}$. A principal polarization $\lambda$
has $\deg(\lambda)=1$. Let $\A=(A=E^g,\lambda)$ be a principally polarized
superspecial abelian variety of dimension $g$ over $\bFp$ with
$\bFp$-isomorphism class $[\A]$.
An $(\ell)^g$-isogeny $\psi:\A=(A=E^g,\lambda)\rightarrow \A'
=(A'=E^g, \lambda')$ has kernel $\ker(\psi)$ a maximal isotropic
subgroup of $A[\ell]$ and $\psi^*(\lambda')=\ell\lambda$.

The theory of $\ell$-isogeny graphs for supersingular 
elliptic curves in characteristic $p$ extends to $(\ell)^g$-isogenies of
principally polarized superspecial abelian varieties
of dimension $g$, giving the 
big isogeny graph $\Gr_{\!g}(\ell,p)$, the little
isogeny graph $\gr_{\!g}(\ell, p)$ and the 
enhanced isogeny graph $\wgr_{\!g}(\ell,p)$; this was the focus of \cite{jz}.
These $(\ell)^g$-isogeny graphs are connected in higher dimensions $g$
as they are for $g=1$ \cite[Thm.~43]{jz}, but unlike the $g=1$ case
the big isogeny graph is not in general Ramanujan for $g>1$ \cite[Sect.~10]{jz}.

This paper further develops the category of $[\ell]$-polarized abelian
varieties introduced in \cite{jz}. Isogenies in this category form
cell complexes, whereas $(\ell)^g$-isogenies of principally polarized 
abelian varieties only form graphs.
A $g$-dimensional $[\ell]$-polarized abelian variety 
$\X=(X,\lambda)$
has a type $\ty (\X)$ with $0\leq \ty (\X)\leq g$ and 
a natural $[\ell]$-dual $\hat{\X}=(\hat{X}=\Pic^0(X), [\lambda])$,
cf.~Remark \ref{pol}.  If $\X'=(X',\lambda')$ is an $[\ell]$-polarized
abelian variety, there is a notion (Section \ref{type}) of $[\ell]$-isogeny 
$f:\X\rightarrow 
\X'$ if $\ty (\X)>\ty (\X')$. Compositions of 
$[\ell]$-isogenies can again be $[\ell]$-isogenies, resulting in Section
\ref{cell}
in an  $[\ell]$-isogeny complex
$\wco_{\ell}(\cA_0)$ once we fix a base principally polarized abelian variety
$\cA_0=(\mathbf{A}_0,\lambda_0)$.  This {\em enhanced isogeny complex} $\wco_{\ell}(\cA_0)$
is a {\em semi-simplicial complex}  in the 
sense of Eilenberg and Zilber \cite{ez}
or {\em $\Delta$-complex} as in \cite{hat}. Most importantly, in Theorem
\ref{pav T} we see that the connected
component of $\wco_{\ell}(\cA_0)$ containing the base $\cA_0$
is a quotient of the Bruhat-Tits building $\cB_g$ of the symplectic
group $\Sp_{2g}(\Q_\ell)$.  

Mapping an $[\ell]$-polarized abelian
variety $\A$ to its dual $\hat{\A}$ and $[\ell]$-isogenies $f:\A\rightarrow \A'$
to their duals $\hat{f}:\widehat{\A'}\rightarrow \hat{A}$ gives an involution
$\iota$ on $\wco_{\ell}(\cA_0)$; the quotient is the {\em little isogeny complex}
$\co_{\ell}(\cA_0)$. Unfortunately, the action of $\iota$ is not
admissible as in Section \ref{ha}, so $\co_{\ell}(\cA_0)$ does not
inherit the structure of a $\Delta$-complex from $\wco_{\ell}(\cA_0)$.
We can avoid this difficulty if we barycentrically subdivide $\wco_{\ell}(\cA_0)$
before taking the quotient by $i$ (then the action is admissible), but this
explodes the number of cells making it too difficult to compute examples.
So we keep the picture we have and examine the quotient by $i$ simplex by 
simplex.  We show that  if $g\leq 4$, then the quotient $\co_{\ell}(\cA_0)$,
although not a $\Delta$-complex, is a CW-complex.  We explain the cell
types we get in this case.

The real power of this approach to isogenies emerges
when we restrict to the case that the base $\cA_0$ is a superspecial 
abelian variety in characteristic $p$. Let $\H=\H_p$ be the definite
quaternion algebra of reduced discriminant $p$ with a maximal order
$\OO_{\H}$.  Then the isogeny
complexes $\wco_{\ell}(\cA_0)$ and $\co_{\ell}(\cA_0)$ are finite
and connected, and they are the quotients of the building $\cB_g$
of the symplectic group $\Sp_{2g}(\Q_\ell)$ by the quaternionic unitary
group $\rU_g(\OO_{\H}[1/\ell])$ and the general quaternionic unitary group
$\GU_g(\OO_{\H}[1/\ell])$, respectively.
Crucially, in this case it is possible to describe the isogeny complexes
completely in terms of hermitian forms over definite quaternion algebras.
Our notion of {\em Brandt complexes} generalizing Brandt matrices is 
developed in Section \ref{brandtc}.  Brandt complexes are amenable to 
machine computation and the paper contains many computational results.

In a forthcoming paper \cite{jz2}, we use the Brandt complex descriptions
of 
\[
\wco_{\ell}(\cA_0)=\rU_{g}(\OO_{\H}[1/\ell])\backslash \cB_g \quad\text{and}
\quad \co_{\ell}(\cA_0)=\GU_g(\OO_{\H}[1/\ell])\backslash \cB_g
\]
 to deduce results on the torsion in 
\[
H^\ast(\rU_{g}(\OO_{\H}[1/\ell]),\Z)\quad\text{and}\quad H^\ast(\GU_g(\OO_{\H}[1/\ell]),\Z)
\]
for a range of examples. We conclude this paper with an extended example,
giving in detail the with the $[2]$-isogeny complexes of superspecial abelian surfaces in characteristic $7$. We return to this example in \cite{jz2}
and compute cohomology.

\section{Superspecial abelian varieties in characteristic \texorpdfstring{$p$}{p}}
\label{super}

\subsection{\texorpdfstring{\except{toc}{\boldmath{$[\ell]$}}\for{toc}{$[\ell]$}}{[\unichar{"2113}]}-polarizations on abelian varieties}
\label{type}

The general reference for this section is \cite[Sect.~6]{jz}.  Proofs,
details, and additional citations to the literature can be found there.

Let $X$ be an abelian variety over a field $k$ (not necessarily algebraically 
closed) with dual abelian variety $\hat{X}=\Pic^{0}(X)$.  The Poincar\'{e}
line bundle on $X\times\hat{X}$ is denoted $\mathcal{P}$. A {\sf polarization}
of $X$ over $k$ is a symmetric isogeny  $\lambda:X\rightarrow \hat{X}$
defined over $k$ such that the line bundle $(1,\lambda)^\ast\mathcal{P}$ is 
ample. The {\sf degree} $\deg(\lambda)$ of a polarization $\lambda:X\rightarrow
\hat{X}$ is $\deg(\lambda)=\#\ker(\lambda)$; this is always a square by
the Riemann-Roch theorem (see \cite[Sect.~16]{Mu}).  The {\sf reduced degree}
of $\lambda$ is $\rdeg(\lambda)=\sqrt{\deg(\lambda )}$. A polarization of
degree $1$ is a {\sf principal polarization}.  If $\lambda:X\rightarrow
\hat{X}$ is a polarization and $\phi:X'\rightarrow X$ is an isogeny, then
\begin{equation}
\label{rower}
\phi^\ast (\lambda)\colonequals \hat{\phi}\circ\lambda\circ \phi:X'\rightarrow
\widehat{X'}
\end{equation}
is a polarization of $X'$ with
\begin{equation}
\label{rowers}
\deg(\phi^\ast(\lambda))=\deg(\lambda)\deg(\phi)^2\qquad\text{and}\qquad
\rdeg(\phi^\ast (\lambda))=\rdeg(\lambda)\deg(\phi).
\end{equation}

Suppose $n\in\NN$ and $(\charr k ,n)=1$ if $\charr k>0$. A polarization
$\lambda$ of $X$ gives rise to the nondegenerate Weil pairing
\begin{equation}
\label{Weil}
\langle\,\,\, ,\,\,\,\rangle_\lambda\colonequals
\langle \,\,\, ,\,\,\,\rangle_{X,\lambda}:\ker(\lambda)\times\ker(\lambda)\rightarrow
\Bmu .
\end{equation}
The more well known pairing
\begin{equation}
\label{radishy}
\langle\,\,\, ,\,\,\,\rangle_{\lambda, n}\colonequals
\langle \,\,\, ,\,\,\,\rangle_{X,\lambda, n}:X[n]\times X[n]\rightarrow
\Bmu_n
\end{equation}
can in fact be given in terms of \eqref{Weil} by
$$\langle\,\,\, ,\,\,\,\rangle_{X,\lambda, n}=\langle\,\,\,
,\,\,\,\rangle_{X,n\circ\lambda}\big|_{X[n]\times X[n]}.$$

The notion of an $[\ell]$-polarization of an abelian variety, 
introduced in \cite[Sect.~6]{jz}, is crucial to what follows
and is summarized below.
\begin{remark1}
\label{pol}
{\rm

Let $\ell$ be a prime such that $\ell\neq\charr k$.
Let $X$ be an abelian variety over the field $k$ with $\dim X=g$.
\begin{enumerate}[\upshape (a)]

\item
\label{pol0}
An $[\ell]$-polarization $\lambda$ on $X$ is 
a polarization such that $\ker(\lambda)\subseteq X[\ell]$.
In particular, the degree of an $[\ell]$-polarization $\lambda$ is 
$\deg(\lambda)=\ell^{2r}$ for $0\leq r\leq g$. We say that 
$r=\log_{\ell}(\rdeg(\lambda))$
is the {\sf type} $\ty(\lambda)$ of the $[\ell]$-polarization $\lambda$.
If $\lambda$ is an $[\ell]$-polarization on $X$ with 
$\ty(\lambda)=r$, we say that
$\X=(X,\lambda)$ is an $[\ell]$-polarized abelian variety of type $r$
and write $\ty(\X)=r$.

\item
\label{pol1.5}
Let $\X=(X,\lambda)$ be an $[\ell]$-polarized abelian variety with 
type $t(\X)=r$.
The {\sf $[\ell]$-dual} of $\X$ is the $[\ell]$-polarized 
abelian variety $\hat{\X}=(\hat{X}=\Pic^{0}(X), [\lambda])$ with 
the composition 
\[
X\stackrel{\lambda}{\longrightarrow}\hat{X}\stackrel{[\lambda]}
{\longrightarrow}X
\]
equal to multiplication by $\ell$. We have
$\ty(\hat{\X})=\hat{r}\colonequals g-r$.
If $\X=(X,\lambda)$ is principally polarized (so of type $0$),
then the $[\ell]$-dual polarization $[\lambda]$ on 
$X\cong\hat{X}$ is $\ell \lambda$.

\end{enumerate}
}
\end{remark1}

\begin{definition}
\label{poll}
{\rm 
Let $\ell$ be a prime such that $\ell\neq\charr k$.
Let $X$ be an abelian variety over the field $k$ with $\dim X=g$.
\begin{enumerate}[\upshape (a)]
\item
\label{pol2a}
Let $\X=(X, \lambda)$ and $\X'=(X',\lambda')$ be polarized abelian
varieties.  An isogeny $f:X\rightarrow X'$ is an {\sf isogeny} from
$\X$ to $\X'$, written $f:\X\rightarrow \X'$, if $\lambda =f^\ast
(\lambda')$.  Now suppose $\X$ and $\X'$ are $[\ell]$-polarized
abelian varieties of types $r$, $s$, respectively.  An isogeny
$f:\X\rightarrow \X'$ is {\sf strict} if $r\neq s$.  If
$f:\X\rightarrow \X'$ is strict, then $r>s$, $\deg(f)=\ell^{(r-s)}$,
and $\lambda:X\rightarrow\hat{X}$ factors through $f$.\\ 
{\bf WARNING:} An $[\ell]$-polarization $\lambda:X\rightarrow \hat{X}$
does {\em not} give an isogeny from $\X$ to $\hat{\X}$.

\item
\label{pol3}

Let $\X=(X,\lambda)$ be an $[\ell]$-polarized abelian variety of type $r$.
A subgroup $0\neq C\subseteq X[\ell]$ of order $\ell^d$, $1\leq d\leq r$,
 is an  {\sf $[\ell]$-subgroup} of $\X$
if $C\subseteq \ker(\lambda)$ and $C$ is isotropic
with respect to the Weil pairing $\langle\,\,\,\, , \,\,\,\rangle_\lambda$.
\end{enumerate}
}
\end{definition}
\begin{proposition}
\label{owl}
Suppose $\X=(X, \lambda)$ and $\X'=(X',\lambda')$ are $[\ell]$-polarized abelian
varieties of types $r$, $s$, respectively. Let $f:\X\rightarrow \X'$
be an isogeny.  Then $\hat{f}:\widehat{X'}\rightarrow \hat{X}$ gives 
an isogeny $\hat{f}:\widehat{\X'}\rightarrow \hat{\X}$. If $f:\X\rightarrow \X'$
is strict, then $\hat{f}:\widehat{\X'}\rightarrow \hat{\X}$ is strict.
\end{proposition}
\begin{proof}
Since $f$ is an isogeny of polarized abelian varieties, by
\eqref{rower} that means $\lambda=\hat{f}\circ\lambda'\circ f$.  Now
by Remark \ref{pol}(\ref{pol1.5}) that means
$[\lambda]\circ\hat{f}\circ\lambda'\circ f=\ell$.  Hence,
$f\circ[\lambda]\circ\hat{f}\circ\lambda'=\ell$, and thus
$\hat{f}^\ast([\lambda])=f\circ[\lambda]\circ\hat{f}=[\lambda']$,
which implies $\hat{f}:\widehat{\X'}\rightarrow \hat{\X}$ is an
isogeny.

Since $\ty(\widehat{\X'})=\hat{s}=g-s$ and 
$\ty(\widehat{\X})=\hat{r}=g-r$, $\hat{f}$ is strict if and only if $f$ is.
\end{proof}
\begin{proposition}
\label{hawk} 
Suppose $\X=(X,\lambda)$ is an $[\ell]$-polarized abelian variety.
Let $C\subseteq X[\ell]$ be an $[\ell]$-subgroup of $\X$ of order
$\ell^d$, $1\leq d\leq r$,  with
$f_C:X\rightarrow X_C\colonequals X/C$
the isogeny taking the quotient by $C$.
Then $X_C =X/C$
has a canonical $[\ell]$-polarization $\lambda_C$ such that the 
composition
\[
X\stackrel{f_C}{\longrightarrow} X_C\stackrel{\lambda_{C}}{\longrightarrow} 
\widehat{X_C}\stackrel{\hat{f}_C}{\longrightarrow}\hat{X}
\]
is $\lambda$. The $[\ell]$-polarized abelian variety
$\X_C=(X_C,\lambda_C)$ has type $\ty(\X_C)=s=r-d$ and the 
isogeny $f_C:\X\rightarrow \X_C$ is strict as in Definition 
\textup{\ref{poll}\eqref{pol2a}}. Moreover a strict 
isogeny $f:\X\rightarrow \X'$ of $[\ell]$-polarized
abelian varieties is of the form $f_C$ with $\X'\cong\X_C$
for $C$ an $[\ell]$-subgroup of $\X$.
\end{proposition}
\begin{proof}
First suppose that $C\subset X[\ell]$ is an $[\ell]$-subgroup of $\X$.
Then since $C$ is isotropic, we have $C\subset
C^\perp\subset\ker\lambda$.  Consider the corresponding
decomposition $$\lambda:X\xrightarrow{f_C} X_C\xrightarrow{\lambda_C}
X/C^\perp\to\hat{C}.$$ By properties of duals, there is an
identification $X/C^\perp=\widehat{X_C}$ such that the map
$X/C^\perp=\widehat{X_C}\to\hat{C}$ is isomorphic to $\hat{f}_C$.
Since $\lambda=\hat{f}_C\circ\lambda_C\circ f_C$ and $\lambda$ is a
polarization, its not hard to see that so is $\lambda_C$.  The formula
for the type follows trivially by counting dimensions.

Now suppose, $f:X\to X'$ be a strict isogeny.  Let $C=\ker f$.  Then
$\hat{f}:\hat{X}'\to\hat{X}$ corresponds to
$C^\perp\subset\ker\lambda$.  Thus $C\subset C^\perp$, hence $C$ is
isotropic and thus an $[\ell]$-subgroup of $\X$.
\end{proof}
A strict isogeny $f:\X\rightarrow \X'$ between $[\ell]$-polarized
abelian varieties 
 will be called an {\sf $[\ell]$-isogeny}.

\subsection{Superspecial abelian varieties}
\label{sup}
A superspecial abelian variety $A/\bFp$ with $\dim A=g$ is isomorphic
to the product of $g$ supersingular elliptic curves.
Fix a supersingular elliptic curve $E/\bFp$ with $\cO=\cO_E=\End(E)$
a maximal order in the rational definite quaternion algebra $\H_p$
ramified at $p$. It is a theorem of Deligne, Ogus, and Shioda that 
such a superspecial $A/\bFp$ is isomorphic to $E^g$.

Fix a prime $\ell$ with $\ell\neq p$.
\begin{definition}
\label{ss}
{\rm

For a dimension $g\geq 1$, let $\sSgpr$ be the set of $\bFp$-isomorphism
classes $[\A]$ of $g$-dimensional $[\ell]$-polarized 
superspecial
abelian varieties $\A=(A,\lambda)$  over $\bFp$ of type $r$, $0\leq r\leq g$.
In case $r=0$, $\sS_{\!g}(p)_0\colonequals\sS_{\!g}(\ell,p)_0$ is the set
of $\bFp$-isomorphism classes of principally polarized superspecial abelian
varieties.
The set $\sSgpr$ is finite; set 
\begin{equation}
\label{daikon}
h_g(\ell,p)_r= \#\sSgpr, \,\,
h_g(\ell,p)=\sum_{r=0}^g h_g(\ell,p)_r,\,\,\text{and }
h_g(p)=h_g(\ell,p)_0=\#\sS_{\!g}(p)_0.
\end{equation}
For $g$ fixed, it is convenient  to define for 
$0\leq r \leq g$:
\begin{equation}
\label{pistol}
\overline{r} =\overline{\hat{r}}= \{r, \hat{r}\}, \text{ viewed as a multiset}.
\end{equation}
Taking the $[\ell]$-dual 
 gives 
a canonical bijection 
\begin{equation}
\label{ghee}
\iota:\sSgpr\stackrel{\simeq}{\longrightarrow} \sSgpgr,\,\,
0\leq r\leq g, \text{ with } \iota([\A]=[\hat{\A}];
\end{equation}
see \cite[Defn.~28]{jz}.
Hence $h_g(\ell,p)_r=h_g(\ell,p)_{\hat{r}}$.   If $g$ and $p$ are fixed,
set 
\begin{equation}
\label{pistol1}
h(\ell;r)\colonequals h_g(\ell,p)_r,\quad h\colonequals h(\ell;0),
\quad\text{and}\quad
h(\ell)\colonequals h_g(\ell,p).
\end{equation}
%\sum_{r=0}^g h(\ell ; r)
%\begin{align}
%\label{bird}
%h(\ell ;r)&\colonequals h_{g}(\ell,p)_r = \#\sSgpr ,\\
%\nonumber h(\ell)&\colonequals
%\sum_{r=0}^g h(\ell ; r),\\
%\nonumber h^{\ast}(\ell)&=\sum_{r=1}^g h(\ell, r), \text{ and }\\
%\nonumber h&=h(\ell ; 0).
%\end{align}

}
\end{definition}
Table \ref{dodge} gives the
$h_g(\ell, p)_r$ which we need for subsequent examples.
\begin{center}
\begin{table}[h]
\begin{tabular}{c|c|c||c|c}
 $g$ & $\ell$ & $p$ &  $h_g(\ell,p)_0$ & 
$h_g(\ell,p)_1$ \\ \hline \hline
$2$ & $2$ & $7$ &  $2$ & $4$\\
$2$ & $2$ & $11$ & $5$ & $10$ \\ 
$3$&   $2$ & $3$ & $2$ & $3$\\
\end{tabular}\\[.25in]
\caption{$h_g(\ell,p)_r=\#\sS_{\!g}(\ell,p)_r
=h_g(\ell,p)_{\hat{r}}= \#\sS_{\!g}(\ell,p)_{\hat{r}}$}
\label{dodge}
\end{table}
\end{center}
Using the notation of \eqref{pistol} and \eqref{ghee}, put
\begin{equation}
\label{water}
\begin{split}
\sSgpr&=\{[\A_1],\ldots , [\A_{h(\ell ;r)}]\}\text{ with }
\A_i=(A_i,\lambda_i)\text{ and }
\sSgpz =\{[\A_1],\ldots , [\A_h]\},\\
\sSgpl & =\coprod_{r=0}^{g}\sSgpr =\{[\A_1],\ldots ,[\A_{h(\ell)}]\},\\
\sS_{\!g}(\ell,p)_{\overline{r}}&=\begin{cases}
\left(\sSgpr\coprod \sSgpgr\right)/\iota\quad\text{if $r\neq \hat{r}$,} \\
\sS_{\!g}(\ell,p)_{g/2}/\iota \quad\text{if $r=g/2$}.
\end{cases}\\
%\cong\sSgpr\cong\sSgpgr ,\\
\osSgpl&=
\coprod_{r=0}^{\lfloor g/2\rfloor} \sS_{\!g}(\ell,p)_{\overline{r}}.
%\coprod_{r=0}^{\lfloor g/2\rfloor}\sSgpr
\end{split}
\end{equation}
\begin{remark1}
\label{swan}
{\rm
We can identify $\overline{\sS}_g(\ell,p)_{\overline{r}}$ with the
set of multisets $[\overline{\A}]\colonequals \{[\A],[\hat{\A}]\}$ with $[\A]\in\sS_{\!g}(\ell,p)_r$.
If $r\neq \hat{r}$, i.e., if $r\neq g/2$, then
$\sS_{\!g}(\ell,p)_{\overline{r}}\cong\sSgpr\cong\sSgpgr$.
Elements of $\sS_{\!g}(\ell,p)_{\overline{r}}$ have type 
$\overline{r}=\{r, \hat{r}\}$.
Let 
\begin{equation}
\label{read}
h_g(\ell,p)_{\overline{r}}=\#\sS_{\!g}(\ell,p)_{\overline{r}}\quad
\text{and}\quad \overline{h}_g(\ell,p)=\#\osSgpl=\sum_{r=0}^{\lfloor g/2\rfloor}
h_g(\ell,p)_{\overline{r}}.
\end{equation}
Then  $h_g(\ell,p)_{\overline{r}}=h_g(\ell,p)_{r}$ if $r\neq g/2$.
In Table \ref{dodge1} we augment Table \ref{dodge} 
by giving the $h_g(\ell,p)_{\overline{r}}$.
\begin{center}
\begin{table}[h]
\begin{tabular}{c|c|c||c|c}
 $g$ & $\ell$ & $p$ &  $h_g(\ell,p)_{\overline{0}}$ & 
$h_g(\ell,p)_{\overline{1}}$ \\ \hline \hline
$2$ & $2$ & $7$ & $2$ & $4$ \\
$2$ & $2$ & $11$ & $5$ & $8$ \\ 
$3$&   $2$ & $3$ & $2$ & $3$\\
\end{tabular}\\[.25in]
\caption{$h_g(\ell,p)_{\overline{r}}=\#\sS_{\!g}(\ell,p)_{\overline{r}}$}

\label{dodge1}
\end{table}
\end{center}
}
\end{remark1}

It is key that the image of an $[\ell]$-polarized superspecial 
abelian variety under a strict isogeny is superspecial:
\begin{proposition}
\label{beaten}
Suppose $\A=(A,\lambda)$ is an $[\ell]$-polarized superspecial abelian
variety  with  type $\ty(\A)=r$  and $C\leq A[\ell]$ is an $[\ell]$-subgroup of $\A$
of order $\ell^d$ , $1\leq d\leq r$,  as in
Definition \textup{\ref{poll}\eqref{pol3}}. Then $A/C$ is a superspecial
abelian variety and hence $\A_C=(A/C,\lambda_C)$ as in 
Proposition \textup{\ref{hawk}} is an $[\ell]$-polarized superspecial abelian
variety with $\ty(\A_C)=s=r-d$.
\end{proposition}
\begin{proof}
This follows from Proposition \ref{hawk} and the fact the image
under a 
separable isogeny of a superspecial abelian variety is superspecial---see
\cite[p.~36]{Oo}.
\end{proof}

We now define two equivalence relations on strict isogenies of 
$[\ell]$-polarized superspecial abelian varieties.
\begin{definition}
\label{monk}
{\rm
Suppose $[\A=(A, \lambda)]\in\sSgpr$ and
$[\A'=(A',\lambda')]\in\sSgps$
with $0\leq s<r\leq g$.  Suppose $f, g:\A\rightarrow \A'$ are strict
isogenies.  
\begin{enumerate}[\upshape (a)]
\item
\label{monk1}
Say  $f\sim_{b}g$ if there exists $\alpha_2\in\Aut(\A')$ such
that $f=\alpha_2 g$. Write $[f]_b$ for the equivalence class with
respect to $\sim_b$ containing $f$.
Put
\[
\Iso_\ell(\A,\A')=\{[f]_b\mid \text{$f:\A\rightarrow \A'$ is a strict
isogeny}\}.
\]
Set
\[
\Iso_\ell(\A)_s=\{[f]_b\in\Iso_\ell(\A,\A'') \text{for some
$\A''\in\sSgps$}\}.
\]
\item
\label{monk2}
Say $f\sim_{l}g$ if there exist $\alpha_1\in\Aut(\A)$ and
$\alpha_2\in\Aut(\A')$ such that $f=\alpha_2 g \alpha_1$.
Write $[f]_l$ for the equivalence class with respect to $\sim_l$
containing $f$.
Put
\[
\iso_\ell(\A,\A')=\{[f]_l\mid \text{$f$ is a strict isogeny from $\A$
to $\A'$}\}.
\]
Set
\[
\iso_\ell(\A)_s=\{[f]_l\in\Iso_\ell(\A,\A'') \text{for some
$\A''\in\sSgps$}\}.
\]
\end{enumerate}
}
\end{definition}
\begin{remark1}
\label{blue}
{\rm
In light of Proposition \ref{hawk}, we could equivalently give
Definition \ref{monk} in terms of the $[\ell]$-subgroups of
Definition \ref{poll}\eqref{pol3}. Suppose $[\A]\in
\sSgpr$ for $0\leq r\leq g$. 
For $0\leq s<r$, set 
\begin{equation}
\label{Iso}
\Iso_\ell(\A)_s=\{C\mid C\text{ is an $[\ell]$-subgroup of $\A$ 
of order $\ell^{r-s}$}\}.
\end{equation}
Define an equivalence relation $\sim$ on $[\ell]$-subgroups of $\A$
by $C\sim C'$ if there exists $\alpha\in\Aut(\A)$ such that
$\alpha(C)=C'$.  Let $[C]$ denote the equivalence class with respect
to $\sim$ containing $C$.  Set
\begin{equation}
\label{robin1}
\iso_\ell(\A)_s=\{[C]\mid C\in \Iso_\ell(\A)_s\}.
\end{equation}
}
\end{remark1}
\begin{proposition}
\label{wren}
Suppose  $[\A]\in\sSgpr$.  Then $\#\Iso_\ell(\A)_s$ neither depends on
$p$ nor $g$ nor on the choice of $[\A]\in\sSgpr$. 
\end{proposition}
\begin{proof}
Let $\A=(A,\lambda)$.  By Definition \ref{poll}(\ref{pol3}) and
\eqref{Iso} $$\Iso_\ell(\A)_s=\{C \mid C\subset\ker(\lambda) \text{ is
  an isotropic subgroup of order $\ell^{r-s}$}\},$$ where
  $\ker(\lambda)$ is an $\ell$-group of rank $2r$ with a
  nondegenerate symplectic pairing.  Thus $\#\Iso_\ell(\A)_s$ depends
  only on $\ell$, $r$, and $s$.
\end{proof}
For $0\leq s<r\leq g$ and $[\A]\in \sSgpr$, set
\begin{equation}
\label{fly}
N(\ell)_{r,s}=\#\Iso_\ell(\A)_s.
\end{equation}
We saw in \cite[(30)]{jz} that 
\begin{equation}
\label{pinky}
N(\ell)_{r,0}=\prod_{k=1}^r(\ell^k+1).
\end{equation}
The general formula for $N(\ell)_{r,s}$ is given below; 
it reduces to \eqref{pinky} when $s=0$. We define $N(\ell)_{r,r}=1$
for all $r\geq 0$.
\begin{proposition}
\label{mouse}
We have
\[
N(\ell)_{r,s}=\prod_{k=0}^{r-s-1}\frac{\ell^{2(r-k)}-1}{\ell^{k+1}-1} .
\]
\end{proposition}
\begin{proof}
By the proof of Proposition \ref{wren}, $N(\ell)_{r,s}$ is equal to
the number of isotropic rank $r-s$ subgroups $C$ of an $\ell$-group $G$ of
rank $2r$ with a nondegenerate symplectic pairing.

If we pick a basis $(v_0,\ldots,v_{r-s-1})$ for $C$, we have
$\ell^{2r}-1$ choices for $v_0\in G\setminus\{0\}$.  Now $v_1$ must be
orthogonal to $v_0$ and independent of $v_0$ so we have
$\ell^{2r-1}-\ell=\ell(\ell^{2(r-1)}-1)$ choices for $v_1$, and in
general we have $\ell^k(\ell^{2(r-k)}-1)$ choices for $v_k$.  Thus
there are $$\prod_{k=0}^{r-s-1}\ell^k(\ell^{2(r-k)}-1)$$ isotropic
rank $r-s$ subspaces $C$ with basis specified.

But each $C$ has $$\prod_{k=0}^{r-s-1}\ell^k(\ell^{k+1}-1)$$ choices
of basis.  Dividing these two values gives the result.
\end{proof}

We use Proposition \ref{mouse} to tabulate $N(\ell)_{r,s}$ for
$0\leq s<r<3$ and $\ell =2$ in Table \ref{dodge2}.
The reader can verify that $N(\ell)_{3m-1,m-1}=N(\ell)_{3m-1,m}$ for $m\geq 1$,
a special case of which is the equality $N(2)_{2,0}=N(2)_{2,1}$
in Table \ref{dodge2}.
\begin{center}
\begin{table}[h]

\begin{tabular}{c|c|c|c|c|c}
$N(2)_{1,0}$ &  $N(2)_{2,0}$ & $N(2)_{2,1}$ &
$N(2)_{3,0}$ & $N(2)_{3,1}$ & $N(2)_{3,2}$\\ \hline \hline
$3$ & $15$ & $15$ & $135$ & $315$ & $63$ \\ 
\end{tabular}\\[.25in]
\caption{\label{bt1} $N(2)_{r,s}$ for $0\leq s<r\leq 3$}
\label{dodge2}
\end{table}
\end{center}

\section{Isogeny graphs of \texorpdfstring{$[\ell]$}{[\unichar{"2113}]}-polarized superspecial abelian varieties}
\label{gen}

Let $\ell$ and $p$ be primes with $p\neq \ell$.
Recall that an $(\ell)^g$-isogeny of a $g$-dimensional 
principally polarized abelian variety is taking the quotient by a maximal
isotropic subgroup of its $\ell$-torsion. In \cite{jz} we defined 
{\sf big, little} and {\sf enhanced} $(\ell)^g$-isogeny graphs of 
$g$-dimensional principally polarized abelian varieties. In the case
of principally polarized superspecial abelian varieties over $\bFp$ it
is possible to construct these graphs from definite quaternion algebras
and hermitian forms over their maximal orders. Such a construction gives
the {\sf Brandt graphs} of \cite[Sect.~5]{jz}.

In this section we relax the requirement of principal polarizations to that of
$[\ell]$-polarizations and consider $[\ell]$-isogeny graphs; see 
Section \ref{type} for definitions.  The definitions of the big, little,
and enhanced $(\ell)^g$-isogeny graph carry over to this more
general setting. The major results of \cite{jz} carry over to 
$[\ell]$-isogeny graphs: We prove in Theorem \ref{ghostly} that they
are connected; we show in Section \ref{gravel} that they are
sometimes (albeit rarely) Ramanujan and sometimes non-Ramanujan; and we
show that the little and enhanced graphs are $\ell$-adically uniformized
by the $1$-skeleton of the Bruhat-Tits building for $\Sp_{2g}(\Q_\ell)$
in Section \ref{unif}.

\subsection{Graphs}
\label{grs}
\begin{definition}
\label{grdef}
{\rm
A (finite) graph $\Grr$ has a finite
set of vertices $\Ver(\Grr)=\{v_1,\ldots , v_s\}$
and a finite set of (directed) edges $\Ed(\Grr)$.  And edge
$e\in\Ed(\Grr)$ has initial vertex $o(e)$ and terminal
vertex $t(e)$.  For vertices $v_i,\,v_j\in \Ver(\Grr)$, put
\[
\Ed(\Grr)_{ij}=\{e\in\Ed(\Grr)\mid o(e)=v_i \text{ and }
t(e)=v_j\}.
\]
The {\sf adjacency matrix} $\Ad(\Grr)\in\Mat_{s\times s}(\Z)$is the 
matrix with 
\[
\Ad(\Grr)_{ij}=\#\Ed({\Grr})_{ij}.
\]
We place no further restrictions on our definition of a graph.
Serre \cite{Ser} requires graphs to be {\sf graphs with opposites}:
every directed edge $e\in\Ed(\Grr)$ has an {\sf opposite}
edge $\overline{e}\in\Ed(\Grr)$.
An edge $e$ with $\overline{e}=e$
is called a {\sf half-edge}. Serre forbids half-edges; we will call a graph
with opposites satisfying his requirements a {\sf graph without half-edges}.
Kurihara \cite{Kur} relaxes Serre's definition to allow
half-edges on a graph with opposites 
giving the notion of a {\sf graph with half-edges}, also 
called an {\sf h-graph} as  in \cite{ijklz1} and \cite{ijklz2}).
A graph with half-edges may have $\emptyset$ as its set of half-edges,
so every graph without half-edges is a graph with half-edges.
}
\end{definition}
\begin{definition}
\label{weight1}
{\rm
\begin{enumerate}[\upshape (a)]
\item
\label{weight11}
A {\sf graph with weights} is a graph with opposites together with a
{\sf weight function} $\w$ mapping vertices and edges to positive integers
that agrees on edges and their opposites and such that for each edge
$e$ 
we have $\w(e)|\w(o(e))$ (which implies $\w(e)|\w(t(e))$ 
since $\w(e)=\w(\overline{e})$).
\item
\label{wad}
The {\sf  weighted adjacency matrix} $A_w:=\Adw(\Grr)$ 
of a graph with weights $\Grr$
with vertices  $\Ver(\Grr)=\{v_1,\ldots , v_s\}$ is
\[
(A_w)_{ij}=\sum_{e\in\Ed({\Grr})_{ij}}\frac{\w(v_i)}{\w(e)}.
\]
\end{enumerate}
}
\end{definition}

\subsection{The enhanced \texorpdfstring{\except{toc}{\boldmath{$[\ell]$}}\for{toc}{$[\ell]$}}{[\unichar{"2113}]}-isogeny graph
\texorpdfstring{\except{toc}{\boldmath{$\wgr_{\!g}([\ell], p)$}}\for{toc}{$\wgr_{\!g}([\ell], p)$}}{tilde gr\unichar{"5F}g([\unichar{"2113}],p)} and its subgraphs
\texorpdfstring{\except{toc}{\boldmath{$\wgr_{\!g}([\ell],p)_{r,s}$}}\for{toc}{$\wgr_{\!g}([\ell], p)_{r,s}$}}{tilde gr\unichar{"5F}g([\unichar{"2113}],p)\unichar{"5F}r,s}
}
\label{enhanced}
The {\sf enhanced} $[\ell]$-isogeny graph
$\wgr_{\!g}([\ell],p)$
has  vertices 
$\Ver(\wgr_{\!g}([\ell],p))=\sSgpl=\coprod_{r=0}^{g}\sSgpr$
in the  notation \eqref{water}. 
In particular, each vertex of $\wgr_{\!g}([\ell],p)$ 
has a type $r$ with $0\leq r\leq g$. 
The vertices $\sS_{\!g}(\ell,p)_0$ of type $0$
and $\sS_{\!g}(\ell,p)_g$ of type $g$ are the {\sf special} vertices.
%which are the vertices of the enhanced $(\ell)^g$-isogeny graph 
%$\wgr_{\!g}(\ell,p)$ of \cite[Sect.~7.3]{jz}.

%It is easiest to think of $\wgr_\ell$ as an undirected graph.  
The 
edges
of $\wgr_{\!g}([\ell],p)$ 
from the vertex $[\A_i]\in\sS_{\!g}(\ell,p)_r
\subseteq\sS_{\!g}(\ell,p)$
to the vertex  $[\A_j]\in\sS_\ell(g,p)_s
\subseteq\sS_{\!g}(\ell,p)$
are
\[
\Ed(\wgr_{\!g}([\ell],p))_{i,j}
=\iso_\ell(\A_i , \A_j )\cup\iso_\ell(\hat{\A}_i,
\hat{\A}_j).
\]
Note that at most one of the sets $\iso_\ell(\A_i ,\A_j)$,
$\iso_\ell(\hat{\A}_i,
\hat{\A}_j)$ is nonempty.
We denote by $\wgr_{\!g}([\ell],p)_{r,s}=\wgr_{\!g}([\ell],p)_{s,r}$ 
the subgraph of $\wgr_{\!g}([\ell],p)$ induced by the vertex
subset 
\[
\sS_{\!g}(\ell,p)_{r}
\cup\sS_{\!g}(\ell,p)_s\subseteq \Ver(\wgr_{\!g}([\ell],p))=
\sS_{\!g}(\ell,p). 
\]
In particular $\Ed(\wgr_{\!g}([\ell],p)_{r,r})=
\emptyset$ for $0\leq r\leq g$.

\begin{remark1}
\label{dirt}
{\rm
The graphs $\wgr_{\!g}([\ell],p)$ and $\wgr_{\!g}([\ell],p)_{r,s}$,
$0\leq r,s\leq g$, are graphs with opposites: if 
$e\in\Ed(\wgr_{\!g}([\ell],p)_{r,s}=\wgr_{\!g}([\ell],p)_{s,r})$ 
corresponds to an isogeny $f$,
then $\overline{e}$
corresponds to the dual isogeny $\hat{f}$.

Define a weight function $\w$ on $\wgr_{\!g}([\ell],p)$ by
$\w([\A])=\#\Aut(A,\lambda)$ for $\A=(A, \lambda)
\in\sSgpl=\Ver(\wgr)$.
On edges put $\w(e)=\#\Aut(f:\A\rightarrow \A')$ if $e$ corresponds to the equivalence
class of the strict isogeny $f:\A\rightarrow\A'$.
The subgraphs $\wgr_{\!g}([\ell],p)_{r,s}$ inherit weights $\w$
from $\wgr_{\!g}([\ell],p)$.

The enhanced isogeny graph $\wgr_{\!g}([\ell],p)$ has a natural
involution.
Let $\iota:\wgr_{\!g}([\ell],p)\rightarrow \wgr_{\!g}([\ell],p)$
and by restriction $\iota:\wgr_{\!g}([\ell],p)_{r,s}
\rightarrow\wgr_{\!g}([\ell],p)_{\hat{s},\hat{r}}$
 be the map with $\iota^2$ the identity defined
on vertices by $\iota([\A])=[\hat{\A}]$ and on edges such
that if $e\in\Ed(\wgr_{\!g}([\ell],p)_{r,s})$ corresponds to the class
$[C]$, then $\iota(e)\in\Ed(\wgr_{\!g}([\ell],p)_{\hat{r},\hat{s}})$ corresponds
to the same class $[C]$.
This is consistent with the definition of
$\iota:\sSgpl\rightarrow\sSgpl$ in Remark \ref{swan}.  The involution
$\iota$ can fix vertices of $\wgr_{\!g}([\ell],p)$.  It does not fix any
edge of $\wgr_{\!g}([\ell],p)$, but can map an edge to its opposite. 
}
\end{remark1}
The weighted adjacency matrix $A\colonequals \Adw(\wgr_{\!g}([\ell],p))$
of Definition \ref{weight1}\eqref{wad} is a block matrix
$A=[A_{r,s}]_{0\leq r,s\leq g}$ with $A_{r,s}$ arising from the edges
going from the vertices $\sS_{\!g}(\ell, p)_r\subseteq \Ver(\wgr_{\!g}([\ell],p))$
to the vertices $\sS_{\!g}(\ell,p)_s$. 
It is convenient to establish the notation that $\zm$ is a matrix
having all entries $0$ with size determined by the context.
We have that $A_{r,s}=A_{\hat{r},\hat{s}}$;
in particular $A_{r,\hat{r}}=A_{\hat{r},r}$.
Moreover $A_{r,r}=\zm$ since
$\Ed(\wgr_{\!g}([\ell],p)_{r,r})=\emptyset$.
The matrix $A_{r,s}$ has size $h_g(\ell,p)_r\times h_g(\ell,p)_s$.
In particular the matrix $A_{r,\hat{r}}$ is square since
$h_g(\ell, p)_r=h_g(\ell, p)_{\hat{r}}$.
The matrix $A_{r,s}$ is a constant row-sum matrix with all the rows adding up to
$N(\ell)_{r,s}$ if $r>s$, $N(\ell)_{\hat{r},\hat{s}}$ if $s>r$, and 
$0$ if $s=r$.
The matrix $\Adw(\wgr_{\!g}([\ell],p)_{r,s})$ is of size
$(h_g(\ell,p)_r + h_g(\ell, p)_s)\times (h_g(\ell, p)_r + h_g(\ell, p)_s)$
if $r\neq s$, and of size $h_g(\ell,p)_r\times h_g(\ell, p)_r$ if $r=s$.
It is a block matrix of the form
\begin{equation}
\label{train}
\Adw(\wgr_{\!g}([\ell],p)_{r,s})=\begin{cases}\begin{bmatrix}
A_{r,r} & A_{r,s}\\
A_{s,r} & A_{s,s}\end{bmatrix}=
\begin{bmatrix}
\zm & A_{r,s}\\
A_{s,r} & \zm\end{bmatrix}\text{ if $r\neq s$,}\\
\begin{bmatrix}A_{r,r}\end{bmatrix}=\zm \text{ if $r=s$.}
\end{cases}
\end{equation}

Hence we have the following proposition.
\begin{proposition}
\label{edict}
\begin{enumerate}[\upshape (a)]
\item
\label{edict1}
The graph $\wgr_{\!g}([\ell],p)_{r,s}$ is bipartite for $0\leq r,s\leq g$.
\item
\label{edict2}
The graph $\wgr_{\!g}([\ell],p)_{r,\hat{r}}$ is a regular graph with regularity
$N(\ell)_{r,\hat{r}}$.
\end{enumerate}
\end{proposition}

\subsection{Two examples of \texorpdfstring{\except{toc}{\boldmath{$\Adw(\wgr_{\!g}([\ell],p))$}}\for{toc}{$\Adw(\wgr_{\!g}([\ell],p))$}}{Ad\unichar{"5F}w(tilde gr\unichar{"5F}g([\unichar{"2113}],p))}}
\label{ex}

\subsubsection{\texorpdfstring{$\Adw(\wgr_{\!2}([2],11))$}{Ad\unichar{"5F}w(tilde gr\unichar{"5F}2([2],11))}}
\label{dented}

Let $A=\Adw(\wgr_{\!2}([2],11))=[A_{r,s}]_{0\leq r,s\leq 2}$.
We have\\[.1in]
\[
A=\begin{bmatrix}
A_{0,0} & A_{0,1}& A_{0,2} \\
A_{1,0} & A_{1,1} & A_{1,2} \\
A_{2,0} & A_{2,1} & A_{2,2} \end{bmatrix} = \scalebox{.92}{%
$\left[ \begin{array}{ccccc|cccccccccc|ccccc}
0 &0& 0& 0& 0& 2& 4& 1& 0& 0& 0 &4& 4& 0& 0& 3& 4& 4& 4& 0\\
0& 0& 0& 0& 0& 0& 9& 0& 0& 0& 0& 0& 0& 6& 0& 9& 3& 3& 0& 0\\
0& 0& 0& 0& 0& 0& 3& 0& 6& 0& 0& 0& 0& 0& 6& 3& 1& 3& 0& 8\\
0& 0& 0& 0& 0& 3& 0& 0& 0& 2& 6& 0& 0& 1& 3& 3& 0& 0& 6& 6\\
0& 0& 0& 0& 0& 0& 0& 0& 6& 0& 6& 0& 3& 0& 0& 0& 0& 4& 3& 8\\
\hline
1& 0& 0& 2& 0& 0& 0& 0& 0& 0& 0& 0& 0& 0& 0& 1& 0& 0& 2& 0\\
1& 1& 1& 0& 0& 0& 0& 0& 0& 0& 0& 0& 0& 0& 0& 1& 1& 1& 0& 0\\
3& 0& 0& 0& 0& 0& 0& 0& 0& 0& 0& 0& 0& 0& 0& 3& 0& 0& 0& 0\\
0& 0& 1& 0& 2& 0& 0& 0& 0& 0& 0& 0& 0& 0& 0& 0& 0& 1& 0& 2\\
0& 0& 0& 3& 0& 0& 0& 0& 0& 0& 0& 0& 0& 0& 0& 0& 0& 0& 3& 0\\
0& 0& 0& 1& 2& 0& 0& 0& 0& 0& 0& 0& 0& 0& 0& 0& 0& 0& 1& 2\\
3& 0& 0& 0& 0& 0& 0& 0& 0& 0& 0& 0& 0& 0& 0& 0& 2& 0& 1& 0\\
1& 0& 0& 0& 2& 0& 0& 0& 0& 0& 0& 0& 0& 0& 0& 0& 0& 2& 1& 0\\
0& 2& 0& 1& 0& 0& 0& 0& 0& 0& 0& 0& 0& 0& 0& 3& 0& 0& 0& 0\\
0& 0& 2& 1& 0& 0& 0& 0& 0& 0& 0& 0& 0& 0& 0& 1& 0& 0& 0& 2\\
\hline
3& 4& 4& 4& 0& 2& 4& 1& 0& 0& 0& 0& 0& 4& 4& 0& 0& 0& 0& 0\\
9& 3& 3& 0& 0& 0& 9& 0& 0& 0& 0& 6& 0& 0& 0& 0& 0& 0& 0& 0\\
3& 1& 3& 0& 8& 0& 3& 0& 6& 0& 0& 0& 6& 0& 0& 0& 0& 0& 0& 0\\
3& 0& 0& 6& 6& 3& 0& 0& 0& 2& 6& 1& 3& 0& 0& 0& 0& 0& 0& 0\\
0& 0& 4& 3& 8& 0& 0& 0& 6& 0& 6& 0& 0& 0& 3& 0& 0& 0& 0& 0
\end{array}
\right]$ }.
\]

\subsubsection{\texorpdfstring{$\Adw(\wgr_{\!3}([2],3))$}{Ad\unichar{"5F}w(tilde gr\unichar{"5F}3([2],3))}}
\label{ex2}

With $g=3$, there are four possible types of vertices: $0$, $1$, $2$, $3$.
We will use the notation of Definition \ref{ss}.
The matrix $A=\Adw(\wgr_{\!3}([\ell],p))$ is size $h_3(\ell,p)\times
h_3(\ell,p)$. It 
is a $4\times 4$ block 
matrix $A=[A_{r,s}]_{0\leq r,s\leq 3}$.  We have $A_{r,r}=\zm$ for 
$0\leq r\leq 3$ and  $A_{r,s}=A_{\hat{r},\hat{s}}$.

From Table \ref{dodge}, we have $h_3(2,3)_0=h_3(2,3)_3=2$
and $h_3(2,3)_1=h_3(2,3)_2=3$.  Let
$A=\Adw(\wgr_3([2],3))= [A_{r,s}]_{0\leq r,s\leq 3}$.  We have\\[.1in]
\[
A= \begin{bmatrix} A_{0,0} & A_{0,1}& A_{0,2} & A_{0,3}\\
A_{1,0} & A_{1,1} & A_{1,2} & A_{1,3}\\
A_{2,0} & A_{2,1} & A_{2,2} & A_{2,3}\\
A_{3,0} & A_{3,1} & A_{3,2} & A_{3,3}\end{bmatrix} =
\left[
\begin{array}{c c|c c c |c c c| c c}
0 & 0 &  9 & 27 & 27  & 45  & 81 & 189  & 81 & 54\\
0 & 0 & 0 & 0 & 63 & 0 & 63 & 252 & 63 & 72\\
\hline
3 & 0 & 0 &0 & 0 &  6 & 9  & 0 & 15 & 0\\
3 & 0 & 0 &0 & 0 &  3& 3  & 9  & 9 & 6\\
1 & 2 & 0 &0 & 0 &  0 & 3 &  12 & 7 & 8\\
\hline
15 & 0 & 6 & 9  & 0 & 0 & 0 & 0 & 3 & 0 \\
9  & 6 & 3 & 3 & 9 & 0 & 0 &0  & 3 & 0\\
7  & 8  & 0 & 3 & 12  & 0 & 0 & 0 & 1 & 2\\
\hline
81 & 54 & 45 & 81 & 189 & 9 & 27 & 27 & 0 & 0\\
63 & 72 & 0 & 63 & 252 & 0 & 0 & 63 & 0 & 0
\end{array}
\right].\\[.01in]
\]

Indeed
the rows of $A_{r,s}$ all add up to $N(3)_{r,s}$ with
$N(3)_{1,0}=N(3)_{2,3}=3$, $N(3)_{2,0}=N(3)_{1,3}=15$,
$N(3)_{3,0}=N(3)_{0,3}=135$,
$N(3)_{3,1}=N(3)_{0,2}=315$, $N(3)_{3,2}=N(3)_{0,1}=63$,
and $N(3)_{0,0}=N(3)_{1,1}=N(3)_{2,2}=N(3)_{3,3}=0$ from Table
\ref{dodge2}.

\subsection{The little \texorpdfstring{\except{toc}{\boldmath{$[\ell]$}}\for{toc}{$[\ell]$}}{[\unichar{"2113}]}-isogeny graph 
\texorpdfstring{\except{toc}{\boldmath{$\gr_{\!g}([\ell], p)$}}\for{toc}{$\gr_{\!g}([\ell], p)$}}{gr\unichar{"5F}g([\unichar{"2113}],p)} and its subgraphs
\texorpdfstring{\except{toc}{\boldmath{$\gr_{\!g}([\ell],p)_{\overline{r},\overline{s}}$}}\for{toc}{$\gr_{\!g}([\ell], p)_{\overline{r},\overline{s}}$}}{gr\unichar{"5F}g([\unichar{"2113}],p)\unichar{"5F}r,s}}
\label{l2}

The {\sf little} $[\ell]$-isogeny graph $\gr_{\!g}([\ell] ,p)$
is the quotient of the enhanced $[\ell]$-isogeny
graph $\wgr_{\!\!g}([\ell],p) $ 
by the involution $\iota$ in Remark \ref{dirt}.
Let $\pi:\wgr_{\!\!g}([\ell],p)\rightarrow \gr_{\!g}([\ell],p)$ 
be the double covering.
In particular 
\begin{equation}
\label{rainy}
\Ver(\gr_{\!g}([\ell],p))=\Ver(\wgr_{\!\!g}([\ell],p))/\iota=
\osSgpl=
\coprod_{r=0}^{\lfloor g/2\rfloor} \sS_{\!g}(\ell,p)_{\overline{r}}
\end{equation}
using the notation \eqref{water}.
The vertices of $\gr_{\!g}([\ell],p)$
have  {\sf type} given by the
multiset $\overline{r}=\{r,\hat{r}\}$ for $0\leq r\leq \lfloor g/2\rfloor$.
For $[\A]\in\sS_{\!g}(\ell, p)$, let $[\overline{\A}]\colonequals
\{[\A],[\hat{\A}]\}$, viewed as a multiset.  Then
\[
\{[\overline{\A}]\mid [\A]\in\sS_{\!g}(\ell,p)\}=\overline{\sS}_{\!g}(\ell,p)=
\Ver(\gr_{\!g}([\ell],p)).
\]

For edges we have $\Ed(\gr_{\!g}([\ell],p))=\Ed(\wgr_{\!\!g}([\ell],p))/\iota$.
The edges
of $\gr_{\!g}([\ell],p)$ going from the vertex 
$[\overline{\A}_i]\in\sS_{\!g}(\ell,p)_{\overline{r}}$ to the 
vertex  $[\overline{\A}_j]\in\sS_{\!g}(\ell,p)_{\overline{s}}$
are
\[
\Ed(\gr_{\!g}([\ell],p))_{ij}=\iso_\ell(\A_i ,\A_j)\cup \iso_\ell(\A_i,\hat{\A}_j)
\cup \iso_\ell(\hat{\A}_i,\A_j)\cup \iso_\ell(\hat{\A}_i,\hat{\A}_j).
\]
Note that at most two of the above sets are nonempty.
If $e\in\Ed(\gr_\ell)_{ij}$, 
the opposite edge 
\[
\overline{e}\in
\Ed(\gr_{\!g}([\ell],p))_{ji}=\iso_\ell(\A_j,\A_i)\cup \iso_\ell(\A_j,\hat{\A}_i)
\cup \iso_\ell(\hat{\A}_j,\A_i)\cup \iso_\ell(\hat{\A}_j,\hat{\A}_i)
\]
is the equivalence class of the dual isogeny.

Define a weight function $\w$  on $\gr$ by descending the weight
function $\w$ on $\wgr_{\!\!g}([\ell],p)$: If 
$\tilde{e}\in\Ed(\wgr_{\!\!g}([\ell],p))$ with
$\pi(\tilde{e})=e\in\Ed(\gr_{\!g}([\ell],p))$, set $\w(e)=\w(\tilde{e})$, unless
$\iota(e)=e$, in which case  $\w(e)=2\w(\tilde{e})$.
Suppose $\tilde{v}\in\Ver(\wgr_{\!\!g}([\ell],p))$ and $\pi(\tilde{v})=v$.
If $\pi$ is unramified at $v$, set $\w(v)=\w(\tilde{v})$.
If $\pi$ is ramified at $v$, set $\w(v)=2\w(\tilde{v})$.

The weighted adjacency matrix $\Adw(\gr_{\!g}([\ell],p))$ is of
size $\overline{h}_{g}(\ell,p)\times \overline{h}_{g}(\ell,p)$.
It is a block matrix of the form
\begin{equation}
\label{rinsed}
\Adw(\gr_{\!g}([\ell],p))=\begin{bmatrix}
A_{\overline{0},\overline{0}} & A_{\overline{0},\overline{1}} &
\cdots & A_{\overline{0},\overline{\lfloor g/2 \rfloor}}\\
A_{\overline{1},\overline{0}} & A_{\overline{1},\overline{1}} &
\cdots & A_{\overline{1},\overline{\lfloor g/2 \rfloor}}\\
% & & & \\
\vdots & \vdots & \cdots & \vdots\\
%& & & \\
A_{\overline{\lfloor g/2 \rfloor},\overline{0}} & A_{\overline{\lfloor g/2 \rfloor},\overline{1}} &
\cdots & A_{\overline{\lfloor g/2 \rfloor},\overline{\lfloor g/2 \rfloor}}
\end{bmatrix}.
\end{equation}

The subgraph $\gr_{\!g}([\ell],p)_{\overline{r},\overline{s}}$ of
$\gr_{\!g}([\ell],p)$ is induced by the vertex subset
\[
\overline{\sS}_{\!g}(\ell,p)_{\overline{r}}\cup
\overline{\sS}_{\!g}(\ell,p)_{\overline{s}}\subseteq
\Ver(\gr_{\!g}([\ell],p))=\overline{\sS}_{\!g}(\ell, p).
\]
We have 
\begin{equation}
\#\Ver(\gr_{\!g}([\ell],p)_{\overline{r},\overline{s}})=\begin{cases}
\overline{h}_{g}(\ell,p)_{\overline{r}}+\overline{h}_{g}(\ell,p)_{\overline{s}}
\quad\text{if $\overline{r}\neq\overline{s}$},\\
\overline{h}_{g}(\ell,p)_{\overline{r}}\quad\text{if }\overline{r}=\overline{s}.
\end{cases}
\end{equation}

The weighted adjacency matrix
$\Adw(\gr_{\!g}([\ell],p)_{\overline{r},\overline{s}})$ 
is a
block matrix
\[
\Adw(\gr_{\!g}([\ell],p)_{\overline{r},\overline{s}})=
\begin{cases}
\begin{bmatrix}
A_{\overline{r},\overline{r}} & A_{\overline{r},\overline{s}} \\
A_{\overline{s},\overline{r}} & A_{\overline{s},\overline{s}}.
\end{bmatrix}\text{if $\overline{r}\neq\overline{s}$},\\[.2in]
\begin{bmatrix} A_{\overline{r},\overline{r}}\end{bmatrix}\text{ if $\overline{r}=
\overline{s}$}
\end{cases}
\]
of size $(h_g(\ell,p)_{\overline{r}} +h_g(\ell,p)_{\overline{s}})\times
(h_g(\ell,p)_{\overline{r}} +h_g(\ell,p)_{\overline{s}})$ if $\overline{r}\neq 
\overline{s}$, $h_g(\ell, p)_{\overline{r}}\times h_g(\ell , p)_{\overline{r}}$
if $\overline{r}=\overline{s}$.
We can give the blocks $A_{\overline{r},\overline{s}}$
in terms of the 
blocks $A_{r,s}$
as in \eqref{train}.  
Let us first suppose that we are in the generic case that neither
of $\overline{r}$, 
$\overline{s}$ is $\{g/2,\,g/2\}$. Then $h_g(\ell,p)_{\overline{r}}=
h_g(\ell,p)_r$ and $h_g(\ell,p)_{\overline{s}}=h_g(\ell,p)_s$.  We have
\begin{equation}
\label{pinto}
\begin{split}
A_{\overline{r},\overline{r}}&=A_{r,\hat{r}}\,,\\
A_{\overline{r},\overline{s}}&= A_{r,s}+A_{r,\hat{s}}\quad\text{if 
$\overline{r}\neq\overline{s}$}\, .
\end{split}
\end{equation}

It is convenient to establish notation for the vertices
$[\A]\in\sS_{\!g}(\ell, p)_{g/2}$ ramified, respectively \'{e}tale,
 in the double cover
$\sS_{\!g}(\ell,p)_{g/2}\rightarrow 
\sS_{\!g}(\ell,p)_{\overline{g/2}}$ which is quotienting by the involution
$\iota$. 
Set
\begin{equation}
\begin{split}
\Ram_{\! g}(\ell,p)_{g/2}&=
\{[\A]\in\sS_{\!g}(\ell, p)_{g/2}\mid [\A]=[\hat{\A}]\}\\
\Et_{\! g}(\ell,p)_{g/2} &=\sS_{\!g}(\ell,p)_{g/2}
\setminus\Ram_{\!g}(\ell,p)_{g/2}=
\{[\A]\in\sS_{\!g}(\ell,p)_{g/2}\mid [\A]\neq [\hat{\A}]\}
\end{split}
\end{equation}
and put
\begin{equation}
r_g(\ell, p)_{g/2} =\#\Ram_{\!g}(\ell,p)_{g/2}\quad\text{and}\quad
e_g(\ell,p)_{g/2}=\# \Et_{\!g}(\ell,p)_{g/2}.
\end{equation}
Then 
\begin{equation}
\label{puppy}
\begin{split}
r_g(\ell,p)_{g/2}+e_g(\ell,p)_{g/2}& =h_g(\ell,p)_{g/2}, \\
r_g(\ell,p)_{g/2}+ (1/2)e_g(\ell,p)_{g/2} &=
h_g(\ell,p)_{\overline{g/2}}
\end{split}
\end{equation}
Suppose  $\overline{r}=\{g/2, g/2\}$ and $\overline{s}\neq\overline{r}$.
\begin{enumerate}[\upshape (a)]
\label{put}
\item
\label{put1}
$A_{\overline{r},\overline{r}}$ is the matrix $\zm$ of size
 $h_g(\ell, p)_{\overline{g/2}}\times
h_g(\ell, p)_{\overline{g/2}}$.
\item
\label{put2}
Suppose $[\overline{\A}]\in\sS_{\!g}(\ell, p)_{\overline{g/2}}$.
\begin{enumerate}[\upshape (i)]
\item
If $[\A]\in \Ram_{\!g}(\ell,p)_{g/2}$, then the $[\overline{\A}]$-row
of $A_{\overline{r},\overline{s}}$ is double the $[\A]$-row of $A_{r,s}$.
\item
If $[\A]\in \Et_{\! g}(\ell,p)_{g/2}$, then the $[\overline{\A}]$-row
of $A_{\overline{r},\overline{s}}$ is the sum of 
the $[\A]$-row of $A_{r,s}$ and the $[\hat{\A}]$-row of $A_{r,s}$.
\end{enumerate}
\item
\label{put3}
Suppose $[\overline{\A}]\in\sS_{\!g}(\ell, p)_{\overline{g/2}}$.
\begin{enumerate}[\upshape (i)]
\item
If $[\A]\in \Ram_{\!g}(\ell,p)_{g/2}$, then the $[\overline{\A}]$-column
of $A_{\overline{s},\overline{r}}$ is the $[\A]$-column of $A_{s,r}$.
\item
If $[\A]\in \Et_{\! g}(\ell,p)_{g/2}$, then the $[\overline{\A}]$-column
of $A_{\overline{s},\overline{r}}$ is the sum of 
the $[\A]$-column of $A_{s,r}$ and the $[\hat{\A}]$-column of $A_{s,r}$.
\end{enumerate}
\end{enumerate}

\subsection{The big \texorpdfstring{\except{toc}{\boldmath{$[\ell]$}}\for{toc}{$[\ell]$}}{[\unichar{"2113}]}-isogeny graph
\texorpdfstring{\except{toc}{\boldmath{$\Gr_{\!g}([\ell], p)$}}\for{toc}{$\Gr_{\!g}([\ell], p)$}}{Gr\unichar{"5F}g([\unichar{"2113}],p)} and its subgraphs 
\texorpdfstring{\except{toc}{\boldmath{$\Gr_{\!g}([\ell],p)_{\overline{r},\overline{s}}$}}\for{toc}{$\Gr_{\!g}([\ell], p)_{\overline{r},\overline{s}}$}}{Gr\unichar{"5F}g([\unichar{"2113}],p)\unichar{"5F}r,s}}
\label{whopper}

Retain the notation of Section \ref{l2}.
The vertices of the big $[\ell]$-isogeny graph $\Gr_{\!g}([\ell],p)$
are
\begin{equation}
\label{cabbage}
\Ver(\Gr_{\!g}([\ell],p))=\Ver(\gr_{\!g}([\ell],p))=
\osSgpl=
\coprod_{r=0}^{\lfloor g/2\rfloor} \sS_{\!g}(\ell,p)_{\overline{r}}.
\end{equation}
The edges of $\Gr_{\!g}([\ell],p)$ going from the vertex
$[\overline{\A}_i]\in\sS_{\!g}(\ell,p)_{\overline{r}}$ to the 
vertex  $[\overline{\A}_j]\in\sS_{\!g}(\ell,p)_{\overline{s}}$
are
\[
\Ed(\Gr_{\!g}([\ell],p))_{ij}=\Iso_\ell(\A_i ,\A_j)\cup \Iso_\ell(\A_i,\hat{\A}_j)
\cup \Iso_\ell(\hat{\A}_i,\A_j)\cup \Iso_\ell(\hat{\A}_i,\hat{\A}_j)
\]
with definitions as in Definition \ref{monk}\eqref{monk1} and Remark \ref{blue}.
Note that at most two of the above sets are nonempty.

The adjacency matrix $\Ad(\Gr_{\!g}([\ell],p))$ is of size 
$\overline{h}_{g}(\ell,p)\times \overline{h}_{g}(\ell,p)$.
It is a block matrix of the form
\begin{equation}
\label{rinsed2}
\Ad(\Gr_{\!g}([\ell],p))=\begin{bmatrix}
{\sf A}_{\overline{0},\overline{0}} & {\sf A}_{\overline{0},\overline{1}} &
\cdots & {\sf A}_{\overline{0},\overline{\lfloor g/2 \rfloor}}\\
{\sf A}_{\overline{1},\overline{0}} & {\sf A}_{\overline{1},\overline{1}} &
\cdots & {\sf A}_{\overline{1},\overline{\lfloor g/2 \rfloor}}\\
% & & & \\
\vdots & \vdots & \cdots & \vdots\\
%& & & \\
{\sf A}_{\overline{\lfloor g/2 \rfloor},\overline{0}} & 
{\sf A}_{\overline{\lfloor g/2 \rfloor},\overline{1}} &
\cdots & {\sf A}_{\overline{\lfloor g/2 \rfloor},\overline{\lfloor g/2 \rfloor}}
\end{bmatrix}.
\end{equation}

The subgraph $\Gr_{\!g}([\ell],p)_{\overline{r},\overline{s}}$ of
$\Gr_{\!g}([\ell],p)$ is induced by the vertex subset
\[
\overline{\sS}_{\!g}(\ell,p)_{\overline{r}}\cup
\overline{\sS}_{\!g}(\ell,p)_{\overline{s}}\subseteq
\Ver(\gr_{\!g}([\ell],p))=\overline{\sS}_{\!g}(\ell, p).
\]
The adjacency matrix $\Ad(\Gr_{\!g}([\ell],p)_{\overline{r},\overline{s}})$ 
is a block matrix
\[
\Ad(\Gr_{\!g}([\ell],p)_{\overline{r},\overline{s}})=
\begin{cases}
\begin{bmatrix}
{\sf A}_{\overline{r},\overline{r}} & {\sf A}_{\overline{r},\overline{s}} \\
{\sf A}_{\overline{s},\overline{r}} & {\sf A}_{\overline{s},\overline{s}}.
\end{bmatrix}\text{if $\overline{r}\neq\overline{s}$},\\[.2in]
\begin{bmatrix} {\sf A}_{\overline{r},\overline{r}}\end{bmatrix}\text{ if $\overline{r}=
\overline{s}$}.
\end{cases}
\]
For $r\neq g/2$ set $\Gr_{\!g}([\ell],p)_r=\Gr_{\!g}([\ell],p)_{\overline{r},
\overline{r}}$.  The graph $\Gr_{\!g}([\ell],p)_r$ is a regular graph
with valence $N(\ell)_{r,\hat{r}}$ if $r>\hat{r}$ and
$N(\ell)_{\hat{r},r}$ if $\hat{r}>r$. In case $r=0$, the graph
$\Gr_{\!g}([\ell],p)_0=\Gr_{\!g}([\ell],p)_g$ is the big isogeny
graph $\Gr_{\!g}(\ell,p)$ of \cite[Sect.~7.1]{jz}.
\begin{theorem}
\label{storm}
Let $A_{\overline{r},\overline{s}}$ be as in \eqref{rinsed}
and ${\sf A}_{\overline{r},\overline{s}}$ be as in \eqref{rinsed2}.
Then $A_{\overline{r},\overline{s}}= {\sf A}_{\overline{r},\overline{s}}$.
In particular, 
$\Ad(\Gr_{\!g}([\ell],p))=\Adw(\gr_{\!g}([\ell],p))$ and
$\Ad(\Gr_{\!g}([\ell],p)_{\overline{r},\overline{s}}=
\Adw(\gr_{\!g}([\ell],p)_{\overline{r},\overline{s}}$.
\end{theorem}

\subsection{Two examples of \texorpdfstring{\except{toc}{\boldmath{$\Adw(\gr_{\!g}([\ell],p))=
\Ad(\Gr_{\!g}([\ell],p))$}}\for{toc}{$\Adw(\gr_{\!g}([\ell],p))=
\Ad(\Gr_{\!g}([\ell],p))$}}{Ad\unichar{"5F}w(gr\unichar{"5F}g([\unichar{"2113}],p))=Ad(Gr\unichar{"5F}g([\unichar{"2113}],p))}}
\label{exam1}
In this section we give the weighted adjacency matrix for the
little isogeny graphs arising from the examples given in Section
\ref{ex}. In particular, we have to apply the discussion
in Section \ref{l2} to calculate the $A_{\overline{r},\overline{s}}$
from the $A_{r,s}$ given in Sections \ref{ex2} and \ref{dented}.

\subsubsection{\texorpdfstring{$\Adw(\gr_{\!2}([2],11))=\Ad(\Gr_{\!2}([2],11))$}{Ad\unichar{"5F}w(gr\unichar{"5F}2([2],11))=Ad(Gr\unichar{"5F}2([2],11))}}
From Table \ref{dodge1}, $h_2(2,11)_{\overline{0}}=3$ and 
$h_2(2,11)_{\overline{1}}=8$.  We have
$e_2(2,11)_1=4$ and $r_2(2,11)_1=6$; observe that the
relations \eqref{puppy} are satisfied.
In the matrix $\Adw(\wgr_2([2],11))$ given in 
Section \ref{dented}, the block $A_{0,1}$ has its first six
columns corresponding to $[\A]\in\sS_{\!2}(2,11)_1$ with $[\hat{\A}]=[\A]$.
The $[\A]\in\sS_{\!2}(2,11)_1$ corresponding to Column 7 in $A_{0,1}$
has $[\hat{\A}]\in\sS_{\!2}(2,11)_1$ corresponding to Column 9.  Similarly
the $[\A]$ corresponding to Column 8 in $A_{0,1}$ has
$[\hat{\A}]$ corresponding to Column 10. There is the analogous
situation in the rows of $A_{1,0}$: Rows 1-6 each correspond to a vertex
fixed by $\iota$, the vertex of Row 7 is mapped under $\iota$ to the
vertex of Row 9, and the vertex of Row 8 is mapped under $\iota$ to
the vertex of Row 10.  This information is sufficient to calculate
$\Adw(\gr_2([2],11))$ using the recipe in Section \ref{l2}.
Let $A=\Ad(\Gr_2([2],11))=\Adw(\gr_2([2],11))=[A_{\overline{r},\overline{s}}]_{0\leq r,s\leq 1}$, cf. Theorem \ref{storm}.
Then\\[.1in]
\begin{equation}
A=\begin{bmatrix} A_{\overline{0},\overline{0}} &
A_{\overline{0},\overline{1}}\\
 A_{\overline{1},\overline{0}} &
A_{\overline{1},\overline{1}}
\end{bmatrix} =
\left[
\begin{array}{ccccc|cccccccc}
3& 4& 4& 4& 0& 2& 4& 1& 0& 0& 0& 4& 4\\
9& 3& 3& 0& 0& 0& 9& 0& 0& 0& 0& 6& 0\\
3& 1& 3& 0& 8& 0& 3& 0& 6& 0& 0& 0& 6\\
3& 0& 0& 6& 6& 3& 0& 0& 0& 2& 6& 1& 3\\
0& 0& 4& 3& 8& 0& 0& 0& 6& 0& 6& 0& 3\\
\hline
2& 0& 0& 4& 0& 0& 0& 0& 0& 0& 0& 0& 0\\
2& 2& 2& 0& 0& 0& 0& 0& 0& 0& 0& 0& 0\\
6& 0& 0& 0& 0& 0& 0& 0& 0& 0& 0& 0& 0\\
0& 0& 2& 0& 4& 0& 0& 0& 0& 0& 0& 0& 0\\
0& 0& 0& 6& 0& 0& 0& 0& 0& 0& 0& 0& 0\\
0& 0& 0& 2& 4& 0& 0& 0& 0& 0& 0& 0& 0\\
3& 2& 0& 1& 0& 0& 0& 0& 0& 0& 0& 0& 0\\
1& 0& 2& 1& 2& 0& 0& 0& 0& 0& 0& 0& 0
\end{array}
\right].
\end{equation}

\subsubsection{\texorpdfstring{$\Adw(\gr_{\!3}([2],3))=\Ad(\Gr_{\!3}([2],3))$}{Ad\unichar{"5F}w(gr\unichar{"5F}3([2],3))=Ad(Gr\unichar{"5F}3([2],3))}}
From Table \ref{dodge1}, $h_{3}(2,3)_{\overline{0}}= 2  $
and $h_{3}(2,3)_{\overline{1}}=3$.
Let $A=\Ad(\Gr_{\!3}([2],3))=\Adw(\gr_{\!3}([2],3))=[A_{\overline{r},\overline{s}}]_{0\leq r,s\leq 1}$, see Theorem \ref{storm}.
Then\\[.1in]
\begin{equation}
A=\begin{bmatrix} A_{\overline{0},\overline{0}} &
A_{\overline{0},\overline{1}}\\
 A_{\overline{1},\overline{0}} &
A_{\overline{1},\overline{1}}
\end{bmatrix} =
\left[
\begin{array}{c c | c c c}
81 & 54 & 54 & 108 & 216\\
63 & 72 & 0 & 63 &  315\\
\hline
18 & 0 & 6 & 9 & 0 \\
12 & 6 & 3 & 3 & 9\\
8 & 10 & 0 & 3 & 12\end{array}
\right].
\end{equation}

\section{First results on the graphs \texorpdfstring{$\wgr_{\!g}([\ell],p)$}{tilde gr\unichar{"5F}g([\unichar{"2113}],p)},
\texorpdfstring{$\gr_g([\ell],p)$}{gr\unichar{"5F}g([\unichar{"2113}],p)},
and
\texorpdfstring{$\Gr_g([\ell],p)$}{Gr\unichar{"5F}g([\unichar{"2113}],p)}}
\label{dung}
For $r\neq g/2$
let $\wgr_{\!g}([\ell],p)_r\colonequals \wgr_{\!g}([\ell],p)_{r,\hat{r}}$
and $\gr_{\!g}([\ell],p)_{\overline{r}}\colonequals \gr_{\!g}([\ell],p)_{
\overline{r},\overline{r}}$.

\subsection{Connectedness results}
We proved in \cite[Sect.~8]{jz} that the graphs
$\wgr_{\!g}([\ell],p)_0$ and $\gr_{\!g}([\ell],p)_{\overline{0}}$ are connected.
Additionally we proved that $\gr_{\!g}([\ell],p)_{\overline{0}}$ is not
bipartite.  The analogous statements are true for our
more general $[\ell]$-isogeny graphs.
\label{shiny}
\begin{theorem}
\label{ghostly}
\textup{1.} The graphs $\wgr_{\!g}([\ell],p)$ and $\gr_{\!g}([\ell],p)$
are connected.\\
\textup{2.} If $r\neq s$, the graph $\wgr_{\!g}([\ell],p)_{r,s}$
is connected.  If $(\overline{r},\overline{s})\neq (\overline{g/2},
\overline{g/2})$,
the graph
$\gr_{\!g}([\ell],p)_{\overline{r},\overline{s}}$ is connected.
\end{theorem}

We will need the following lemma.
\begin{lemma}
Consider the lattice of isotropic subgroups of a symplectic space of
dimension $2g$.  Now consider the subgraph consisting of subspaces of
dimension $r$ and $s$ with $0\le r<s\le g$, i.e., there is an edge
between two spaces if one is contained in the other.  That subgraph is
connected.
\end{lemma}
\begin{proof}
Let $V$ and $V'$ be any two isotropic subspaces of dimension $s$, it
suffices to prove there is a path between them, so suppose that there
isn't.  Let $W=V\cap V'$ with $t=\dim W$ and suppose $t$ is maximal
among all $s$-dimensional subspaces that don't have a path between
them.  We must clearly have $t<r\le s-1$.

Now let $x\in V\setminus W$ and consider $U=x^\perp\cap V'$.  Since
$V$ is isotropic, $W\subset U$.  And $\dim U$ is either $s$ or $s+1$,
in the latter case replace $U$ by an $s$-dimensional subset containing
$W$.  Set $V''=U\oplus\langle x\rangle$.  Now $\dim V\cap V''=t+1$ and
$\dim V'\cap V''=s-1>t$, so by assumption $V''$ is connected to both
$V$ and $V'$ and we are done.
\end{proof}

\begin{proof}[Proof of Theorem \ref{ghostly}]
Part 1 follows trivially from \cite{jz}[Theorem 43] since by
Proposition \ref{hawk} every superspecial abelian variety with an
$[\ell]$-polarization has an isogeny to a type $0$ abelian variety.

We will prove part 2 is stages.

The connectedness of $\wgr_{\!g}([\ell],p)_{g,0}$ is just \cite{jz}[Theorem 43].

Now to prove the connectedness of $\wgr_{\!g}([\ell],p)_{r,0}$ it
suffices to show that if $[\A']-[\A]-[\A'']$ is a subgraph of
$\wgr_{\!g}([\ell],p)_{g,0}$ with $\A$ of type $g$ and $\A'$, $\A''$ of
type $0$ we can finding a path in $\wgr_{\!g}([\ell],p)_{r,0}$ between
$[\A']$ and $[\A'']$.  Let $\A=(A,\lambda)$, by Proposition \ref{hawk} $\A'$
and $\A''$ correspond to isotropic rank-$g$ subgroups of
$\ker(\lambda)$ and rank-$(g-r)$ subgroups of $\ker(\lambda)$ give type
$r$ vertices of $\wgr_{\!g}([\ell],p)_{r,0}$.  Since the graph of
isotropic subspaces of $\ker(\lambda)$ of ranks $g$  and $g-r$
is connected, we have a path
from $[\A']$ to $[\A'']$ in $\wgr_{\!g}([\ell],p)_{r,0}$.

Next apply the involution we see that the graphs
$\wgr_{\!g}([\ell],p)_{g,s}$ are connected.  And by the same logic as
above we get the connectedness of $\wgr_{\!g}([\ell],p)_{r,s}$ for
general $r$ and $s$.

Finally, $\gr_{\!g}([\ell],p)_{\overline{r},\overline{s}}$ is
connected since it's the quotient of the connected graph 
$\wgr_{\!g}([\ell],p)_{r,s}$.
\end{proof}

\subsection{The Ramanujan condition}
\label{ghost}
Let $A=\Adw(\wgr_{\!g}([\ell],p))$; it is a block matrix with
$A=[A_{r,s}]_{0\leq r,s\leq g}$.  The matrices $A_{r,\hat{r}}$ are symmetric
with constant row sum equal to  $N(\ell)_{r,\hat{r}}$ if $r>\hat{r}$,
$N(\ell)_{\hat{r},r}$ if $\hat{r}>r$, and $0$ if $r=\hat{r}$.
Let $\Gr_{\!g}([\ell],p)_{r}$ be the regular graph with adjacency matrix
$A_{r,\hat{r}}$ for $r\ne g/2$.
The graph $\Gr_{\!g}([\ell],p)_0=\Gr_{\!g}([\ell],p)_g$ 
is the big isogeny graph $\Gr_{\!g}(\ell,p)$
for principally polarized superspecial abelian varieties
defined in \cite[Sect.~7.1]{jz}.
We saw in \cite[Sect.~10]{jz} that this graph is sometimes, but rarely,
Ramanujan and usually non-Ramanujan.  We get the same behavior
for the graphs $\Gr_{\!g}([\ell],p)_r$ for $0<r<g$.  We give a
Ramanujan and a non-Ramanujan example below.
\subsubsection{A non-Ramanujan example\textup{:} \texorpdfstring{$\Gr_{\!3}([3],2)_1$}{Gr\unichar{"5F}3([3],2)\unichar{"5F}1}}
\label{gravel}
We have
\[
\Ad(\Gr_{\!3}([3],2)_1)=\Adw(\wgr_{\!3}([3],2)_{1,2})=
\begin{bmatrix}
8 & 32 & 0\\
4 & 4 & 32\\
0 & 12 & 28\end{bmatrix}.
\]
This matrix has eigenvalues $40$, $\pm 12$.  As $12<2\sqrt{39}$, the graph
$\Gr_{\!3}([3],2)_1$ is Ramanujan.

\subsubsection{A Ramanujan example\textup{:} \texorpdfstring{$\Gr_{\!3}([2],3)_1$}{Gr\unichar{"5F}3([2],3)\unichar{"5F}1}}
\label{gravel1}
The graph $\wgr_{\!3}([2],3)$ was considered in Section
\ref{ex2}.  We have
\[
\Ad(\Gr_{\!3}([2],3)_1)=\Adw(\wgr_{\!3}([2],3)_{1,2} )=
\begin{bmatrix}
6 & 9 & 0\\
3 & 3 & 9\\
0 & 3 & 12\end{bmatrix}.
\]
The eigenvalues of this matrix are $15$ and $3\pm 3\sqrt{3}$.  Since
$3+3\sqrt{3}>2\sqrt{14}$, the graph $\Gr_{\!3}([2],3)_1$ is not
Ramanujan.

\section{The \texorpdfstring{$\ell$}{\unichar{"2113}}-adic uniformization of
\texorpdfstring{$\wgr_{\!g}([\ell],p)$}{tilde
  gr\unichar{"5F}g([\unichar{"2113}],p)} and
\texorpdfstring{$\gr_{\!g}([\ell],p)$}{gr\unichar{"5F}g([\unichar{"2113}],p)}}
\label{unif}

Let $\cB_{2g}$ be the Bruhat-Tits building for $\GSp_{2g}$ over 
$\Q_\ell$. Let $\Sk=\Sk_{2g}$ be the $1$-skeleton of $\cB_{2g}$.
Vertices $v\in\Sk\subseteq \cB_{2g}$ have a type $\ty(v)\in\Z$
with $0\leq \ty(v)\leq g$.
The {\sf special vertices} are the vertices $v$ with $\ty(v)=0$ or
$\ty(v)=g$.
To handle types of positive-dimensional cells, set
\begin{equation}
\label{gored}
\mathscr{T}_s\colonequals \{(r_0,\ldots, r_{s})\in\Z^{s+1}\mid
g\geq r_0>r_1>\cdots >r_{g}\geq 0\}
\end{equation}
for $0\leq s\leq g$.
Edges $e\in\Sk\subseteq \cB_{2g}$
have a type $\ty(e)=\mathbf{r}\in\mathscr{T}_1$.
For $r,s \in\mathscr{T}_0$, $r\neq s$, let $\Ver(\Sk)_{r,s}=\{v\in\Ver(\Sk)\mid
\ty(v)=r\text{ or }\ty(v)=s\}$ with $\Sk_{r,s}$ the induced subgraph
of $\Sk$ on the vertex set $\Ver(\Sk)_{r,s}$.

Let $R$ be a commutative ring and $M$ be an $R$-algebra with
an anti-involution $x\mapsto x^\dagger$.  
Define the {\sf unitary group} $\rU(M)$ and {\sf general unitary group} 
$\GU_R(M)$ by
\begin{align}
\label{locust}
\rU(M) &= \{x\in M\mid x^\dagger x=1\}\\
\nonumber \GU_R(M)&=\{x\in M\mid x^\dagger x\in R^\times\}.
\end{align}
Recall that $E/\bFp$ is a supersingular elliptic curve with 
$\cO=\cO_E=\End(E)$. We apply \eqref{locust} with $R=\Z[1/\ell]$
and $M=\Mat_{g\times g}(\cO_E[1/\ell])$ with anti-involution 
$x^\dagger =\overline{x}^t$ for $x\in M$.  Set
\begin{align}
\label{locust1}
\rU_g(\cO[1/\ell])&= \rU(\Mat_{g\times g}(\cO[1/\ell]))\quad\text{and}\\
\nonumber \GU_g(\cO[1/\ell])&= \GU_{\Z[1/\ell]}(\Mat_{g\times g}(\cO[1/\ell])).
\end{align}
The groups $\rU_g(\cO[1/\ell])$ and $\GU_g(\cO[1/\ell])$
act on $\cB_{2g}$ with finite quotients. 
\begin{theorem}
\label{drain}
\begin{enumerate}[\upshape (a)]
\item
\label{drain0}
$\rU_g(\cO[1/\ell])\backslash \Sk_{2g}\cong\wgr_{\!g}([\ell],p)$ as 
graphs with weights.

\item
\label{drain1}
The group $\rU_g(\cO[1/\ell])$ acting on $\Sk=\Sk_{2g}$ stabilizes
the subgraph $\Sk_{r,s}$ for $r,s\in\mathscr{T}_0$, $r\neq s$, and
$\wgr_{\!g}([\ell],p)_{r,s}\cong \rU_g(\cO[1/\ell])\backslash \Sk_{r,s}$
as graphs with weights.

\end{enumerate}
\end{theorem}
\begin{proof}
This will follow as a special case of Theorem \ref{pav T}\eqref{U pav}
and the fact that the action of $\rU_g(\cO[1/\ell])$ preserves types.
\end{proof}
\noindent
Theorem \ref{drain}\eqref{drain0} implies that the graph
$\wgr_{\!g}([\ell],p)$ is connected, which we have already established
in Theorem \ref{ghostly}(1).  Likewise Theorem \ref{drain}\eqref{drain1} 
would imply
that $\wgr_{\!g}([\ell],p)_{r,s}$, $r\neq s$, is connected if we knew
that $\Sk_{r,s}$ was connected---but we already know that 
$\wgr_{\!g}([\ell],p)_{r,s}$, $r\neq s$, is connected 
from Theorem \ref{ghostly}(2).
\begin{theorem}
\label{sink}
\begin{enumerate}[\upshape (a)]
%\item
%\label{sink0}
$\GU_g(\cO[1/\ell])\backslash \Sk_{2g}\cong\gr_{\!g}([\ell],p)$ as 
graphs with weights.

%\item
%\label{sink1}
%The group $\GU_g(\cO[1/\ell])$ acting on $\Sk=\Sk_{2g}$ stabilizes
%the subgraph $\Sk_{r,s}$ for $r,s\in\mathscr{T}_0$ and
%$\gr_{\!g}([\ell],p)_{\overline{r},\overline{s}}\cong 
%\GU_g(\cO[1/\ell])\backslash \Sk_{r,s}$
%as graphs with weights.\\
%\bruce{Again, what are the exceptional cases we have to exclude?}
\end{enumerate}
\end{theorem}
\begin{proof}
This will follow as a special case of Theorem \ref{pav T}\eqref{GU pav}.
\end{proof}
In particular, Theorem \ref{sink} implies that the graphs
$\gr_{\!g}([\ell],p)$ and
$\gr_{\!g}([\ell],p)_{\overline{r},\overline{s}}$ are
connected. However, as above, connectedness was already shown in Theorem
  \ref{ghostly}(1).

\section{Isogeny complexes of \texorpdfstring{[$\ell$]}{[\unichar{"2113}]}-polarized abelian varieties}
\label{cell}

Fix a prime $\ell$.
In this Section \ref{cell} and Section \ref{quotient}, we work with isogenies of degree 
$\ell^n$ of   general polarized abelian
varieties $\A=(A,\lambda)$ over an algebraically closed field $k$
with $(\charr k,\ell)=1$.  We then return to superspecial
abelian varieties in Section \ref{donut}.

We will define $[\ell]$-isogeny complexes of 
$[\ell]$-polarized abelian varieties, generalizing the
$[\ell]$-isogeny graphs of Section \ref{gen}.

\subsection{\texorpdfstring{\except{toc}{\boldmath{$\Delta$}}\for{toc}{$\Delta$}}{\unichar{"0394}}-complexes and \texorpdfstring{\except{toc}{\boldmath{$\Delta$}}\for{toc}{$\Delta$}}{\unichar{"0394}}-complexes with half-faces}
\label{tolled}

The complexes $[\ell]$-isogenies naturally form are
 the higher-dimensional analogues
of multigraphs, and quotients of these complexes by an involution.
They are a mild generalization of simplicial complexes
  introduced by Eilenberg and Zilber \cite{ez} in 1950 as
{\sf semi-simplicial complexes}.
Hatcher \cite[Sect.~2.1]{hat} calls them {\sf $\Delta$-complexes}
and we follow this in the interest of brevity and disambiguation.  
However, we use the  purely combinatorial 
definition of Eilenberg-Zilber \cite[Sect.~1]{ez}.
Taking the quotient of a $\Delta$-complex
by an involution gives a {\sf $\Delta$-complex with half-faces} or an
{\sf h-$\Delta$-complex}.  A one-dimensional $\Delta$-complex is a 
multigraph, called simply a graph in Section \ref{grs}.
 A one-dimensional $\Delta$-complex with half-faces or a
one-dimensional 
h-$\Delta$-complex is a graph with half-edges or
an h-graph
as in Section \ref{grs}.

\subsubsection{\texorpdfstring{$\Delta$}{\unichar{"0394}}-complexes}
\label{dated}

\begin{definition}
\label{simp}
{\rm
\begin{enumerate}[\upshape (a)]
\item
\label{simp1}
(\cite[Sect.~1]{ez})  A {\sf $\Delta$-complex}  $\cK$ of dimension $g$ 
consists of a set $\Sigma=\Sigma(\cK)=\{\sigma\}$ of simplexes together with two
functions on $\Sigma$.  The first assigns to any $\sigma\in\Sigma$
an integer $j=\dim(\sigma)$ with 
$g\geq j\geq 0$ called the dimension of $\sigma$.
Put $\Sigma_{j}=\Sigma_j(\cK)=\{\sigma\mid \dim(\sigma)=j\}\subseteq \Sigma$ 
and say that $\sigma\in \Sigma_j$ is a $j$-simplex.
The $0$-simplexes $\Sigma_0(\cK)$ are called
the {\sf vertices} $\Ver(\cK)$ of $\cK$.
Likewise the $1$-simplexes $\Sigma_1(\cK)$
are called the {\sf edges} $\Ed(\cK)$ of $\cK$. We assume
$\Sigma_g\neq\emptyset$;
the simplexes $\sigma\in\Sigma_g(\cK)$ are called {\sf facets}.

The second function on $\Sigma$ associates to each $j$-simplex $\sigma$ and to each integer
$0\leq i\leq j$ a $(j-1)$-simplex $\Fa_i(\sigma)$ called the $i^{\rm th}$
face of $\sigma$ such that 
$\Fa_i(\Fa_{i'}(\sigma))=\Fa_{i'-1}(\Fa_{i}(\sigma))$
for $i<i'$ and $j>1$.  
Iteration of this map gives lower-dimensional faces.
For $0\leq i_1< \cdots < i_n\leq j$ define inductively
\[
\Fa_{i_1,\ldots ,i_n}(\sigma)=\Fa_{i_1}(\Fa_{i_2,\ldots, i_n}(\sigma)),
\]
which is a $(j-n)$-simplex.  If $0\leq k_0<\cdots <k_{j-n}\leq j$
is the set  complementary to $\{i_1,\ldots, i_n\}\subseteq\{1, 2, \ldots, j\}$,
then we also write $\Fa_{i_1, \ldots, i_n}(\sigma)=
\sigma_{(k_0,\ldots, k_{j-n})}$.
In particular for $\sigma\in\Sigma_j$ we have $\sigma_{(i)}\in\Sigma_0$
for $0\leq i\leq j$ called the $i^{\rm th}$ vertex of $\sigma$.
The definition of the $i^{\rm th}$ face does not exclude the possibility
of $\sigma,\tau\in\Sigma_j$, $\sigma\neq \tau$, with
$\Fa_i(\sigma)=\Fa_i(\tau)$ for $0\leq i\leq j$ (which would
not be allowed in a simplicial complex).  Hence
a $\Delta$-complex is analogous to a multigraph.

\item
\label{simp2}
A {\sf weighted $\Delta$-complex} is a $\Delta$-complex
together with a {\sf weight function} $\w$ mapping simplexes
to positive integers such that if $\sigma$ is a simplex and 
$\sigma'$ is a face of $\sigma$ then $\w(\sigma)|\w(\sigma')$.
\end{enumerate}
}
\end{definition}

\subsubsection{\texorpdfstring{$\Delta$}{\unichar{"0394}}-complexes with half-faces}
\label{ha}
Suppose $\cK$ is a $\Delta$-complex and $\iota:\cK\rightarrow \cK$
is an involution.  As in \cite[Sect.~IX.10]{br}, say that
$\iota$ is {\sf admissible} if $\iota$ fixing a simplex $\sigma$ of $\cK$
is equivalent to $\iota$ fixing $\sigma$ pointwise.
If the involution $\iota$ of the $\Delta$-complex $\cK$ is admissible,
then the quotient $\cK/\iota$ is naturally again a $\Delta$-complex.
However, if the action is not admissible, then the quotient
$\cK/\iota$ will in general have half-faces and will not be a
$\Delta$-complex.
We refer to the images of simplexes in $\cK$ under the natural
projection $\cK\rightarrow \cK/\iota$ as cells.
 Many references only consider admissible actions
since this is obtainable by passing to a
subdivision:  Let $\cK'$ be the first barycentric subdivision of $\cK$.
Then $\iota$ induces an involution $\iota:\cK'\rightarrow \cK'$
which is admissible.  However, passing to the barycentric subdivision
explodes the number of cells, making computing examples too difficult.
Hence we do not pass to the barycentric subdivision, and so we 
have to deal directly with half-faces.  We draw some of the half-faces
in dimensions $1$, $2$, and $3$ which occur in Section \ref{pari}.

\subsection{The enhanced \texorpdfstring{\except{toc}{\boldmath{$[\ell]$}}\for{toc}{$[\ell]$}}{[\unichar{"2113}]}-isogeny complex 
\texorpdfstring{\except{toc}{\boldmath{$\wco_\ell(\cA_0)$}}\for{toc}{$\wco_\ell(\cA_0)$}}{tilde co\unichar{"5F}\unichar{"2113}(A0)}}
\label{enha}

The isogenies we will consider are of the following type:
\begin{definition}
{\rm
Two polarized abelian varieties $\X=(X, \lambda)$ and
$\X'=(X',\lambda')$ are said to be {\sf $\ell$-power isogenous} if
there exists an isogeny in the sense of Definition
\ref{poll}\eqref{pol2a} from $\ell^n\X=(X,\ell^n\lambda)$ to $\X'$ 
for some $n\in\NN$.  We will say that they're {\sf evenly} (respectively, {\sf
  oddly}) isogenous if $n$ is even (respectively, odd).
Note that a principally polarized abelian variety is oddly isogenous to
its $[\ell]$-dual  by Remark \ref{pol}\eqref{pol1.5}
since $1$ is odd.
}
\end{definition}

Pick a base $g$-dimensional principally polarized abelian variety
$\cA_0=(\mathbf{A}_0, \lambda_0)$ over an 
algebraically closed field $k$ with
$\ell$ coprime to $\charr k$; $\mathbf{A}_0$ is not required to be superspecial.
All polarized abelian
varieties discussed in this section and the next will be polarized and
$\ell$-power isogenous to $\cA_0$.

We will first define the enhanced isogeny complex 
$\wco_{\ell}(\cA_0)$.
\begin{definition}
\label{enh}
{\rm
\begin{enumerate}[\upshape (a)]
\item
\label{enh1}

The vertices $\Ver(\wco_{\ell}(\cA_0))$ or $0$-simplexes 
$\Sigma_0(\wco_{\ell}(\cA_0))$
 correspond to isomorphism classes $[\A=(A,\lambda)]$
of abelian varieties $A$ with a 
polarization $\lambda$ of type $r$ with  $0\le r\le g$ as in 
Remark \ref{pol}\eqref{pol0} such that $\A$ is $\ell$-power
isogenous to $\cA_0$. Recall that a polarization of type $0$ is simply
a principal polarization. If $\A=(A,\lambda)$ with $\lambda$ a polarization
of type $g$, then there is a principal polarization $\lambda'$ of $A$
with $[\A=\ell\A'=(A,\ell\lambda')]$ 
for $\A'=(A,\lambda')$
as in Remark \ref{pol}\eqref{pol1.5}.
Thus there is a natural bijection between polarizations
of type $g$ on $A$ and those of type $0$.
Nevertheless, $[\A]$ and $[\A']$ 
correspond to distinct $0$-cells of the enhanced complex.  
The weight of a vertex $[\A]\in\Sigma_0(\wco_{\ell}(\cA_0))$ is 
$\#\Aut(\A)$.

Suppose $\A_i=(A_i, \lambda_i)$ with $\lambda_i$ a polarization of type $r_i$.
For $r_0>r_1$ the (unoriented) edges $\Sigma_1(\wco_{\ell}(\cA_0))$ 
are isomorphism classes of strict maps 
$f:\A_0\rightarrow \A_1$ as in Definition \ref{poll}\eqref{pol2a},
where an isomorphism between two strict maps is defined in the natural manner.  Such an edge will be said to be of {\sf type $(r_0,r_1)$}.  
If $e\in\Sigma_1(\wco_\ell(\cA_0))$ is the 
edge corresponding to $f$, then the faces of $e$ are the vertices
$\Fa_0(e)= [\A_1]=e_{(1)}\in\Sigma_0(\wco_{\ell}(\cA_0))$
and $\Fa_1(e)=[\A_0]=e_{(0)}\in\Sigma_0(\wco_{\ell}(\cA_0))$.
The weight $\w(e)$ of the edge $e$ is the cardinality of the automorphism group of $f:\A_1\rightarrow \A_2$.

At height $k$, a $k$-simplex $\sigma$ is an isomorphism class of sequences 
of strict maps 
\[
\A_0\xrightarrow{f_1}\A_1\xrightarrow{f_2}\A_2\xrightarrow{f_3}\cdots\xrightarrow{f_{k-1}}\A_{k-1}
\xrightarrow{f_{k}}\A_{k}.
\]
Necessarily $g\geq r_0>r_1>\cdots >r_{k-1}>r_{k}\geq 0$.
Such a $k$-simplex will be said to be of 
{\sf type } $\bfr=(r_0,\ldots,r_{k})$. 
We denote by $\Sigma_{\bfr}\subseteq
\Sigma_k$ the set of $k$-cells  of type
$\bfr=(r_0,\ldots, r_{k})$, $g\geq r_0>r_1>\cdots >r_{k-1}>r_{k}\geq 0$.
The $(k+1)$ {\sf faces} of $\sigma\in\Sigma_{\bfr}$
are 
the $(k-1)$-simplexes  as follows:
\begin{enumerate}[\upshape (a)] 
\item
The face $\Fa_0(\sigma)$ of type $(r_1, r_2,\,\ldots \, ,r_{k})$
is the isomorphism class of the sequence of strict maps
\[
\A_1\xrightarrow{f_2}\A_2\xrightarrow{f_3}\cdots\xrightarrow{f_{k-1}}\A_{k-1}\xrightarrow{f_k}\A_{k}.
\]
\item
The face $\Fa_i(\sigma)$ for $1\leq i\leq k-1$ of type $(r_0,\ldots, 
\widehat{r}_i,\ldots, r_{k})$ (here $\widehat{r}_i$ means omit $r_i$)
is the isomorphism class of the sequence of strict maps
\[
\A_0\xrightarrow{f_1}\A_1\xrightarrow{f_2}\A_2\xrightarrow{f_3}\cdots 
\xrightarrow{f_{i-1}}\A_{i-1}
\xrightarrow{f_{i+1}\circ f_{i}}\A_{i+1}\xrightarrow{f_{i+2}}\cdots
\xrightarrow{f_{k-1}}\A_{k-1}
\xrightarrow{f_k}\A_{k}.
\]  
\item
The face $\Fa_{k}(\sigma)$ of type $(r_0,\ldots  ,r_{k-1})$
is the isomorphism class of the sequence
of strict maps
\[
\A_0\xrightarrow{f_1}\A_1\xrightarrow{f_2}\cdots
\xrightarrow{f_{k-1}}\A_{k-1}.
\]
\end{enumerate}
The weight $\w(\sigma)$  of the $k$-simplex $\sigma$
 is the cardinality of the automorphism group of the sequence of maps.
In particular that there are $\binom{g+1}{k+1}$ types of
$k$-simplexes.\\
{\bf NOTE:} The $\Delta$-complex
$\wco_{\ell}(\cA_0)$ is connected if and only if $\cA_0=(\mathbf{A}_0,\lambda_0)$
is oddly $\ell$-isogenous to itself, otherwise it has two
connected components consisting of those polarized abelian varieties
that are evenly and
oddly $\ell$-isogenous to $\cA_0$.  Let $\wco^0_{\ell}(\cA_0)$ be the
connected component of $\wco_{\ell}(\cA_0)$ containing the
vertex  $[\cA_0]$. The vertices of $\wco^0_\ell(\cA_0)$ are all isomorphism
classes of polarized abelian varieties evenly $\ell$-power isogenous
to $\cA_0$.
\item
\label{dual}
The {\sf dual map} $\iota:\wco_{\ell}(\cA_0)\rightarrow
\wco_{\ell}(\cA_0)$ is the involution sending each simplex to its dual:
a vertex $[\A]$ is sent to its dual $[\hat{\A}]$; observe that if $\A$ was
of type $r$ then $\hat{\A}$ will be of type $\hat{r}=g-r$.  Each higher
simplex
$\A_0\xrightarrow{f_1}\A_1\xrightarrow{f_2}\cdots\xrightarrow{f_{k-1}}\A_{k-1}
\xrightarrow{f_k}\A_{k}$
of type $\bfr=(r_0,\ldots,r_{k})$ gets sent to
$\hat{\A}_{k}\xrightarrow{\hat{f}_k}\hat{\A}_{k-1}\xrightarrow{\hat{f}_{k-1}}\cdots\xrightarrow{\hat{f}_2}\hat{\A}_1\xrightarrow{\hat{f}_1}\hat{\A}_0$
of type $\hbr\colonequals (g-r_{k},\ldots,g-r_0)$.  Observe that this reverses any
orientation of the simplex precisely when $k\equiv1,\,2\pmod{4}$.
\\ {\bf NOTE:} If $\wco_{\ell}(\cA_0)$ is disconnected, then $\iota$
swaps the two components.

\end{enumerate}
}
\end{definition}

\subsection{The little \texorpdfstring{\except{toc}{\boldmath{$[\ell]$}}\for{toc}{$[\ell]$}}{[\unichar{"2113}]}-isogeny graph 
\texorpdfstring{\except{toc}{\boldmath{$\co_\ell(\cA_0)$}}\for{toc}{$\co_\ell(\cA_0)$}}{co\unichar{"5F}\unichar{"2113}(A0)}}
\label{lit}

We define the little $[\ell]$-isogeny complex as the quotient
$\co_\ell(\cA_0)\colonequals \wco_\ell(\cA_0)/\iota$.
Every $k$-cell of $\co_\ell(\cA_0)$ has a type
$\obr=\{\bfr,\hbr \}$ with $\obr=\overline{\hbr}$.
Unfortunately, under this
map some faces are self-dual without every face of their boundary
being self-dual, resulting in what we call {\sf half-faces}.  
Because of this, $\co_{\ell}(\cA_0)$ does not inherit the structure
of a $\Delta$-complex from $\wco_{\ell}(\cA_0)$ -- it is a 
$\Delta$-complex with half-faces, or an h-$\Delta$-complex.

In a future paper \cite{jz2} we consider the cohomology
of $\wco_{\ell}(\cA_0)$ and $\co_{\ell}(\cA_0)$
when $\cA_0$ is a principally polarized
superspecial abelian variety in characteristic $p\neq\ell$.
The cohomology of $\wco_{\ell}(\cA_0)$ is straightforward --
Eilenberg and Zilber \cite{ez} proved that the canonical cochain
complex constructed from the simplexes computes the cohomology of 
a $\Delta$-complex.  But the cohomology of $\co_{\ell}(\cA_0)$
is much trickier.  Topologically the right thing to do is clearly
to barycentrically subdivide all the simplexes of $\wco_{\ell}(\cA_0)$
 before taking the quotient by the dual map $\iota$.  But this greatly
increases the number of cells, making computing examples out of reach.
So we will deal with the half-faces directly.  We examine the low-dimensional
cells in $\co_{\ell}(\cA_0)$ in the next Section \ref{pari} 
and find that the $k$-cells in $\co_{\ell}(\cA_0)$ are balls for $k\leq 4$,
so that $\co_{\ell}(\cA_0)$ is naturally a CW-complex if the dimension $g\leq 4$.
We can therefore compute its cohomology in this case from a canonical
cochain complex constructed from its cells.

\subsection{Half-faces of low dimension}
\label{pari}

In this section we describe what half-faces look like in low dimension.
For each example, we consider the involution $\iota$ which reverses
the order of the vertices in a standard simplex
 with the induced action on each simplex
in the picture.

A $0$-dimensional face is a vertex or point, a $0$-dimensional
half-face is also functionally a point.

A $1$-dimensional face is an edge. A $1$-dimensional half-face is a
half-edge as described and drawn in \cite{Kur} and subsequently
\cite{ijklz1}, \cite{ijklz2}, \cite{jz}:\\[.2in]

\begin{center}
\scalebox{1.2}{
\begin{tikzpicture}
%\draw [red, thick, middlearrow={>}] (1.6,2.2) -- (0,0);
%\draw [red, thick, middlearrow={>}] (1.6,2.2) -- (3.2,0);
\draw (-1,0) -- (1,0);
\draw[dashed] (0,0.3) -- (0,-0.3);
\fill[black] (-1,0) circle (.08cm);
\fill[black] (1,0) circle (.08cm);
\end{tikzpicture}
}
$\quad\qquad$\raisebox{1.4ex}{$\longrightarrow$}$\quad\qquad$
\scalebox{1.2}{
\begin{tikzpicture}
%\draw [red, thick, middlearrow={>}] (1.6,2.2) -- (0,0);
%\draw [red, thick, middlearrow={>}] (1.6,2.2) -- (3.2,0);
\draw (-1,0) -- (0,0);
\draw[dashed] (0,0.3) -- (0,-0.3);
\fill[black] (-1,0) circle (.08cm);
\fill[white] (1,0) circle (.08cm);
\end{tikzpicture}.
}
\end{center}
\vspace*{.2in}
Here the original complex has $2$ vertices and $1$ edge.  The
quotient complex by $\iota$ has $1$ vertex and $1$ half-edge.
Note that the half-edge is contractible to its vertex.

A $2$-dimensional face is a triangle; it has $1$ $2$-simplex,
$3$ edges, and $3$ vertices.  The quotient by $\iota$ in turn has
$1$ half-$2$-simplex, $1$ edge, $1$ half-edge, and $2$ vertices.
 A $2$-dimensional half-face is a
half-triangle as seen below:\\[.1in]

\begin{center}
\scalebox{1.2}{
\begin{tikzpicture}
%\draw [red, thick, middlearrow={>}] (1.6,2.2) -- (0,0);
%\draw [red, thick, middlearrow={>}] (1.6,2.2) -- (3.2,0);
\draw (-1,0) -- (1,0);
\draw (-1,0) -- (0,1.73);
\draw (1,0) -- (0,1.73);
\draw[dashed] (0,2.03) -- (0,-0.28);
\fill[black] (-1,0) circle (.08cm);
\fill[black] (1,0) circle (.08cm);
\fill[black] (0,1.73) circle (.08cm);
\end{tikzpicture}
}
$\quad\qquad$\raisebox{7ex}{$\longrightarrow$}$\quad\qquad$
\scalebox{1.2}{
\begin{tikzpicture}
%\draw [red, thick, middlearrow={>}] (1.6,2.2) -- (0,0);
%\draw [red, thick, middlearrow={>}] (1.6,2.2) -- (3.2,0);
\draw (-1,0) -- (0,0);
\draw (-1,0) -- (0,1.73);
\draw[dashed] (0,2.03) -- (0,-0.28);
\fill[black] (-1,0) circle (.08cm);
\fill[white] (1,0) circle (.08cm);
\fill[black] (0,1.73) circle (.08cm);
\end{tikzpicture}.
}
\end{center}
\vspace*{.1in}
Note that the half-face and its associated half-edge are contractible
to its full edge.

It's also possible for a full $2$-face to have a half-edge as 
shown below:\\[.2in]

\begin{center}
\scalebox{1.2}{
\begin{tikzpicture}
%\draw [red, thick, middlearrow={>}] (1.6,2.2) -- (0,0);
%\draw [red, thick, middlearrow={>}] (1.6,2.2) -- (3.2,0);
\draw (-1,0) -- (1,0);
\draw (-1,0) -- (0,1.73);
\draw (1,0) -- (0,1.73);
\draw (-1,0) -- (0,-1.73);
\draw (1,0) -- (0,-1.73);
\node at (0,0) {\large $\circlearrowleft$};
\fill[black] (-1,0) circle (.08cm);
\fill[black] (1,0) circle (.08cm);
\fill[black] (0,1.73) circle (.08cm);
\fill[black] (0,-1.73) circle (.08cm);
\fill[black] (0,0) circle (.03cm);
\end{tikzpicture}
}
$\quad\qquad$\raisebox{11.4ex}{$\longrightarrow$}$\quad\qquad$
\raisebox{3ex}{\scalebox{1.2}{
\begin{tikzpicture}
%\draw [red, thick, middlearrow={>}] (1.6,2.2) -- (0,0);
%\draw [red, thick, middlearrow={>}] (1.6,2.2) -- (3.2,0);
\draw (0,-1) .. controls (1.1,-0.6) and (1.1,0.5) .. (0,1.73);
\draw (0,-1) .. controls (-1.1,-0.6) and (-1.1,0.5) .. (0,1.73);
\draw (0,-1) -- (0,0);
\fill[black] (0,-1) circle (.08cm);
\fill[black] (0,1.73) circle (.08cm);
\end{tikzpicture}.
}}
\end{center}
\vspace*{.2in}
The quotient has $1$ $2$-face, $2$ edges, $1$ half-edge, and $2$
vertices.
Note that the half-edge is contractible leaving the triangle a digon.

A $3$-dimensional face is a tetrahedron with the involution acting as
rotation about the indicated axis which passes through the midpoint
$A$ of an edge and the midpoint $B$ of the opposing edge. We label
where these midpoints go in the quotient by $\iota$ using the same letter.
 A $3$-dimensional half-face is a
half-tetrahedron as seen below:

\begin{center}
\scalebox{1.2}{
\includegraphics[scale=0.3]{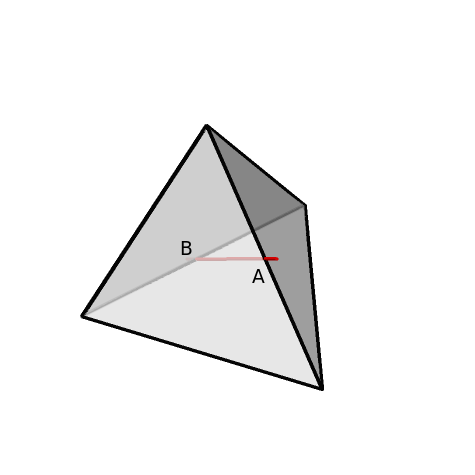}
\raisebox{12ex}{$\longrightarrow$}
\includegraphics[scale=0.3]{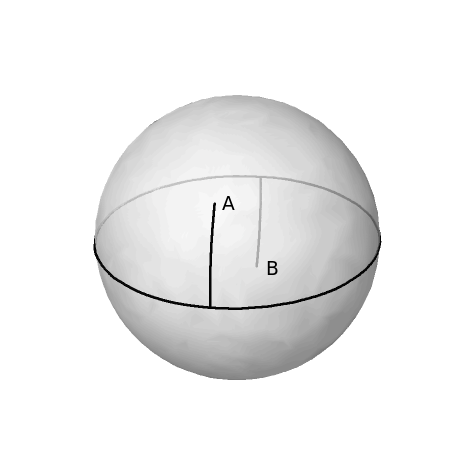}\raisebox{5ex}{.}
}
\end{center}
The original simplex has $1$ $3$-simplex, $4$ $2$-simplexes, $6$ edges,
and $4$ vertices.  The quotient has $1$ half-$3$-simplex, $2$ 
$2$-simplexes, $2$ edges, $2$ half-edges, and $2$ vertices.
Note that the boundary of the half tetrahedron contains half-edges and
faces that turn into digons when the half-edges are contracted as
described above.  The half-tetrahedron itself is still a topological
$3$-ball $B^3$, and it is not contractible to its boundary.

Once we hit faces of dimension $4$ and higher, it is no longer really possible
to draw pictures.  The {\sf standard} $k$-simplex $\Delta_k$ in $\R^{k+1}$ is
\[
\Delta_k =\{(x_1,\ldots, x_{k+1})\mid \sum_{i=1}^{k+1}x_i=1, x_i\geq 0\text{ for }
1\leq i\leq k+1\}.
\]
In this standard $k$-simplex, the involution $\iota$ reversing
the order of the vertices is induced by the linear involution of
$\R^{k+1}$ given in the standard basis by the matrix $M=M_{k+1}$ with
$1$'s on the antidiagonal and $0$'s elsewhere.  The eigenvalues of
the matrix $M_{k+1}$ are $(k+1)/2$ $+1$'s, $(k+1)/2$ $-1$'s if $k$ is odd
and $k/2 +1$ $+1$'s, $k/2$ $-1$'s if $k$ is even. The vector $\vec{v}\in\R^{k+1}$
with all entries $1$ is an eigenvector of $M_{k+1}$ with eigenvalue $1$
and $\vec{v}$ is perpendicular to the affine hyperplane 
$\sum_{i=1}^{k+1}i=1$.  Hence the eigenvalues of $M_{k+1}$ acting on $\Delta_k$
are $(k-1)/2$ $+1$'s, $(k+1)/2$ $-1$'s if $k$ is odd and
$k/2$ $+1$'s, $k/2$ $-1$'s if $k$ is even.

In particular  on $\Delta_4$ the involution acts on $\Delta_4$
 with $2$ $+1$ eigenvalues and 
$2$ $-1$ eigenvalues.  Hence the quotient of $\Delta_4$ by $\iota$
is $\text{(the cone on }\RP^1)\times B^2$, the cone on $\RP^1$ coming
from the $2$ $-1$ eigenvalues and the $2$-ball $B^2$ coming from
the $2$ $+1$-eigenvalues.  But $\RP^1$ is just the circle $S^1$,
so  the half $4$-simplex $\Delta_4/\iota$
is homeomorphic to $\textup(\text{the cone on } S^1\textup)\times B^2$, which is just
$B^2\times B^2\cong B^4$.  Hence a half $4$-simplex is still a 
topological ball $B^4$ and it not contractible to its boundary.

Now on the standard $5$-simplex $\Delta_5$ the involution acts with 
eigenvalues $3$ $-1$'s and $2$ $+1$'s.  As above this gives us that
a half $5$-simplex is homeomorphic to $S\colonequals\textup{(}\text{the cone on }\RP^2\textup{)}\times B^2$.
However, we claim that now in dimension $5$ this is {\em{not}} 
the ball $B^5$; here is one way to see this:
\begin{proposition}
\label{denver}
Let $S=\textup{(}\text{\rm the cone on }\RP^2\textup{)}\times B^2$.
Then $S$ is not homeomorphic to $B^5$, and in fact $\partial S$ is not homotopic to $\partial B^5$.
\end{proposition}
\begin{proof}
The first statement clearly follows from the second.   We claim that 
\[
H_3(\partial S,\Z)\neq 0, \text{ whereas }
H_3(\partial B^5,\Z)=0.
\]
Firstly, recall that $\pi_1(\RP^2)=\Z/2\Z$.  But then $\partial S$ is 
homotopic to the double suspension $\Sigma^2(\RP^2)$.  Since each
suspension moves homology up by $1$ we have that $H_3(\Sigma^2(\RP^2),\Z)\neq 0$
since $H_1(\RP^2,\Z)\neq 0$. But $H_3(\partial B^5,\Z)=0$ since $\partial B^5=S^4$.
\end{proof}
We summarize our discussion above in the following theorem.
\begin{theorem}
\label{rated}
Let $\cA_0$ be a $g$-dimensional principally polarized abelian variety
over an algebraically closed field $k$ with $\charr k\neq \ell$.
Let $K$ be the $\Delta$-complex $\wco_\ell(\cA_0)$ and
$\iota$ the involution of $K$ given by taking isogenies to their
duals and taking $[\ell]$-polarized abelian varieties to their
$[\ell]$-duals.  Set $K/\iota=\co_\ell(\cA_0)=\wco_\ell(\cA_0)/\iota$.
\begin{enumerate}[\upshape (a)]
\item
The half-faces of $K/\iota$ of dimensions $1$ and $2$ are homeomorphic
to the standard cells $B^1$ and $B^2$, and are contractible to their 
boundary.
\item
The half-faces of $K/\iota$ of dimensions $3$ and $4$ are homeomorphic
to the standard cells $B^3$ and $B^4$, and are {\bf{not}}
contractible to their boundary.
\item
The half-faces of $K/\iota$ of dimension $5$ and higher are
{\bf{not}} homeomorphic to the standard cell \textup{(}i.e., the ball of the
same dimension\textup{)}.
\item
Suppose the dimension $g\leq 4$.  Then $\co_\ell(\cA_0)=K/\iota$ can
be given the structure of a CW-complex $K^\ast$, after homotpoically
collapsing the h-$1$-cells and h-$2$-cells.  The CW-complex $K^\ast$
with $0$-cells the $0$-cells of $K/\iota$, $1$-cells the regular edges
of $K/\iota$, $2$-cells the regular $2$-simplexes of $K/\iota$,
$3$-cells equal to the regular $3$-simplexes and half $3$-simplexes of
$K/\iota$, and $4$-cells equal to the regular $4$-simplexes and half
$4$-simplexes of $K/\iota$.

\end{enumerate}
\end{theorem}

\section{Identifying isogeny complexes for \texorpdfstring{$\ell$}{\unichar{"2113}}-polarized
abelian varieties as quotients of the building for
\texorpdfstring{$\Sp_{2g}(\Q_\ell)$}{Sp2g(\unichar{"211A}\unichar{"2113})}}
\label{quotient}

Let ($V$, $\langle\,\, ,\,\,\rangle$) be a $2g$-dimensional symplectic space
over $\Q_\ell$.  Let $W\subset V$ be a $\Z_\ell$-sublattice.  Recall
that there exists a basis of $W$ such that the symplectic form
decomposes as the sum of $g$ forms of the form $\begin{pmatrix} 0 &
   \ell^m \\ -\ell^m & 0 \end{pmatrix}$ with the $g$-tuple
$(m_i)=(m_1,m_2, \ldots, m_g)$ of $m$'s uniquely determined by $W$
(up to permutation).  We will call
$W$ an {\sf $(m_i)$-lattice}.  Also recall that the {\sf dual lattice}
$W^*$ is defined as $\{x\in V \mid \langle W, x \rangle \subset
\Z_\ell \}$.  If $W$ is an $(m_i)$-lattice then $\ell^aW$
is a $(2a+m_i)$-lattice and $W^*$ is a $(-m_i)$ lattice.  A
$((m+1)^n,(m)^{g-n})$-lattice will be called {\sf special} and a
$((1)^n,(0)^{g-n})$-lattice will be called {\sf extra
   special} of {\sf type} $n$.

Recall that the vertices of the building $\cB_g$ for
$\Sp_{2g}(\Q_\ell)=\Sp(V)$ are given by the extra special lattices.
We shall say a vertex corresponding to a $((1)^r,(0)^{g-r})$-lattice
$W$ has {\sf type} $r$.  Recall also that the $k$-faces of the
building are given by chains of vertices with proper inclusions
 $W_0\subsetneq
W_1\subsetneq\cdots\subsetneq W_{k}$ where we  call
$\bfr=(r_0,r_1,\ldots,r_{k})$ the {\sf type} of the face if $r_i$ is
the type of $W_i$. Note that $r_0>r_1>\cdots >r_k$.

Recall that $\Sp(V)$ acts on all these faces and hence the building
in the obvious way preserving the types of faces.  We shall take this
action to be on the left.

Finally recall that we can extend this action to an action of
$\GSp(V)$.  First identify $W$ with the {\sf generalized homothety}
set $\{\ell^aW\}\cup\{\ell^aW^*\}$.  Notice that: this operation
partitions the lattices into equivalence classes; if $W$ is special,
all the elements of the set are; and each special set has precisely
one extra special element.  Also the action of $\GSp(V)$ preserves
generalized homothety.  This induces an action of $\GSp(V)$ on the
building.  However, unlike the $\Sp(V)$ action which preserves types,
the elements of $\GSp(V)$ which scale by an odd power of $\ell$ send
faces of type $\bfr=(r_0,r_1,\ldots,r_{k})$ to faces of type
$\hat{\bfr}\colonequals (g-r_{k},g-r_{k-1},\ldots,g-r_0)$.  
Also since scalar multiplication
fixes generalized homothety, this action factors through $\PGSp(V)$.

\begin{theorem} \label{pav T} \textup{(cf. \cite[Theorem 21]{jz}.)}
Let $\cA_0=(\mathbf{A}_0,\lambda_0)$ be a principally 
polarized abelian variety of dimension $g$
over an algebraically closed field $k$ of characteristic $\charr k$
and $\ell\ne\charr k$ be a prime.  In the notation of \eqref{locust} and 
Definition 
\textup{\ref{enh}\eqref{enh1}} we have
\begin{equation}\label{U pav}
\wco^0_{\ell}(\cA_0)\cong\rU(\End(\mathbf{A}_0)[1/\ell])\bs\cB_g
\end{equation}
and
\begin{equation}\label{GU pav}
\co_{\ell}(\cA_0)\cong\GU(\End(\mathbf{A}_0)[1/\ell])\bs\cB_g .
\end{equation}
\end{theorem}
\begin{proof}
Let $W=\Ta_\ell(\mathbf{A}_0)$ be the Tate module of $\mathbf{A}_0$, and let
$V=\Ta_\ell(\mathbf{A}_0)\otimes\Q_\ell$, both equipped with the symplectic Weil
pairing.
%Identify $\GSp_{2g}(\Q_\ell)=\GSp(V)$ and
%$\GSp_{2g}(\Z_\ell)=\GSp(T)$.
Note that we have an exact sequence $$0\to W\to
V\xrightarrow{\pi}\mathbf{A}_0[\ell^\infty]\to0.$$ This gives us induced
injections $\rU(\End(\mathbf{A}_0)[1/\ell])\to\Sp(V)$ and
$\GU(\End(\mathbf{A}_0)[1/\ell])\to\GSp(V)$.

Let $\A_1$ be evenly $\ell$-isogenous to $\cA_0$ via
$\phi:\ell^{2m}\cA_0\to \A_1$.  We will associate to the pair
$(\phi,\A_1)$ the lattice $W_1=\ell^m\pi^{-1}(\ker \phi)\subset V$.
Since $\A_1$ is $[\ell]$-polarized, we can see that $W_1$ is extra
special of the same type as $\A_1$ by chasing through the matrices.
Conversely, if $W_1$ is extra special and we can pick an $m\in\NN$
such that $\ell^{-m}W_1\supset W$, then $\pi(\ell^{-m}W_1)$ is the
kernel of an $\ell$-power isogeny $\phi:\ell^{2m}\cA_0\to \A_1$ whose
image is $[\ell]$-polarized of the same type as $W_1$.  Furthermore,
picking a different representative corresponds to composing $\phi$
with multiplication by a scalar power of $\ell$.

Note that if $\psi:\A_0\to \A_1$ is an isogeny of type $(r_0,r_1)$,
$r_0>r_1$,
then 
\[
W_1=\ell^m\pi^{-1}(\ker \psi\circ\phi)\supsetneq W_0
\] 
is an 
extra-special lattice of type $r_1$.  Hence the chain $W_0\subsetneq W_1$ is an
edge of type $(r_0,r_1)$.  The preceding can clearly be applied to
each isogeny in a chain to obtain higher faces.  Conversely, if the
chain $W_0\subsetneq W_1$ is an edge of type $(r_0,r_1)$ with
$\ell^{-m}W_0\supset W$ and $\pi(\ell^{-m}W_i)$ the kernel of
$\phi_i:\ell^{2m}\cA_0\to \A_i$, we can decompose
$\phi_1=\psi\circ\phi_0$ with $\psi:\A_0\to \A_1$ an isogeny of
type $(r_0,r_1)$.  Again the preceding can be applied to higher faces
to obtain a chain of isogenies.

Now we show that $W^\prime$ and $W^{\prime\prime}$ correspond to
isomorphic $[\ell]$-polarized abelian varieties if and only if
$W^\prime=\psi W^{\prime\prime}$ for some
$\psi\in\rU(\End(\mathbf{A}_0)[1/\ell])$.  First assume we are given $\psi$.
Pick an $m\in\NN$ such that $\ell^m\psi\in\End(\mathbf{A}_0)$, and an $n$ such
that $\ell^{-n}W^\prime,\,\ell^{-n}W^{\prime\prime}\supset W$.
Therefore
$\pi(\ell^{-n}W^\prime)=\ell^m\psi(\pi(\ell^{-n}W^{\prime\prime}))$.
Hence if $\ker \phi^\prime=\pi(\ell^{-n}W^\prime)$ and $\ker
\phi^{\prime\prime}=\pi(\ell^{-n}W^{\prime\prime})$, then $\phi^\prime
\circ \ell^m\psi = \phi^{\prime\prime}$ and both have the same codomain.

Conversely, suppose $\phi^\prime$ and $\phi^{\prime\prime}$ (using
scalings $m'$ and $m^{\prime\prime}$) induced from $W^\prime$ and
$W^{\prime\prime}$ have the same codomain.  Let $\phi$ be such that
$\phi\circ\phi^\prime=\ell^{2m'}$ let
$\psi=\ell^{-m'-m^{\prime\prime}}\phi\circ\phi^{\prime\prime}$.  Note
that $\psi\in\rU(\End(\mathbf{A}_0)[1/\ell])$ since it preserves the polarization
up to scaling, and chasing through we see that the scale factor is
$1$.  Chasing through the equations we furthermore see that $W^\prime=\psi
W^{\prime\prime}$.

The same logic can be applied to the higher faces thus showing
\eqref{U pav}.

We will now show \eqref{GU pav}.  First suppose we are given
$\psi\in\GU(\End(\mathbf{A}_0)[1/\ell])$ such that $[W^\prime]=[\psi
  W^{\prime\prime}]$ as generalized homothety classes.  If
$\ell^aW^\prime=\psi W^{\prime\prime}$, then $\psi$ is a scalar
multiple of an element of $\rU(\End(\A)[1/\ell])$.  Since that case
was dealt with above, we may assume that $\ell^aW^{\prime*}=\psi
W^{\prime\prime}$.  After possible scaling $W^\prime$,
$W^{\prime\prime}$, and $\psi$ by powers of $\ell$ we may assume that
$\psi\in\End(\A)$, $W^{\prime*}=\psi W^{\prime\prime}$, and
$W^{\prime*},\,\ell^bW^\prime,\,W^{\prime\prime}\supset W$.  Therefore
$\pi(W^{\prime*})=\psi(\pi(W^{\prime\prime}))$.  Hence if $\ker
\phi=\pi(W^{\prime*})$, $\ker \phi^\prime=\pi(\ell^bW^\prime)$, and $\ker
\phi^{\prime\prime}=\pi(W^{\prime\prime})$, then $\phi \circ \psi =
\phi^{\prime\prime}$ and both have the same codomain which is the dual
of the codomain of $\phi^\prime$.

Conversely, if $\phi^\prime$ and $\phi^{\prime\prime}$ induced from
$W^\prime$ and $W^{\prime\prime}$ have dual codomains, let
$\psi=\widehat{\phi^\prime}\circ \phi^{\prime\prime}$.  Note that
$\psi\in\GU(\End(\mathbf{A}_0)[1/\ell])$ since it preserves the polarization up
to scaling, and chasing through we see that the scale factor is $1$.
And chasing through the equations we see that $\ell^aW^{\prime*}=\psi
W^{\prime\prime}$ for some $a\in\Z$.

As before, the same logic can be applied to the higher faces thus showing
\eqref{GU pav}.

\end{proof}

\section{Mass formulas}
\label{donut}

We now return to the special case of superspecial abelian varieties.
If the principally polarized base
abelian variety $\cA_0=(\mathbf{A}_0,\lambda_0)$ has $\mathbf{A}_0$
superspecial in characteristic $p$, then for primes $\ell\neq p$ the
isogeny complexes $\wco_{\ell}(\cA_0)$ and $co_{\ell}(\cA_0)$ enjoy
important properties not necessarily present in the general case:
\begin{enumerate}[\upshape (a)]
\item
In the superspecial case, the isogeny complexes are finite.
\item
In the superspecial case, the isogeny complexes are connected.
\item
In the superspecial case, the isogeny complexes can be described in
terms of hermitian forms over definite quaternion algebras -- this is 
developed in our notion of {\sf Brandt complexes} in Section
\ref{brandtc}.  In particular, there are mass formulas for the isogeny
complexes arising from the definite quaternion algebra description; this
is the subject of the present Section \ref{donut}.
\end{enumerate}

Computing the number of simplexes of a given type in the isogeny
complex of a superspecial abelian variety is in general hard; however,
it's relatively easy to compute their mass.

\begin{definition}
The {\sf mass} of a set $S$ of $k$-simplexes in a weighted simplicial
complex is equal to $$\sum_{\sigma\in S}\frac{1}{w(\sigma)}.$$

Let $\cA_0$ be a principally polarized superspecial abelian variety of
dimension $g$ in characteristic $p$.  We will let
$m_g([\ell],p)_{\bfr}$ be the mass of the set of
$k$-simplexes of type $\bfr=(r_0,r_1,\ldots,r_{k})$ in the enhanced
$[\ell]$-isogeny complex $\wco^0_{\ell}(\cA_0)$.
\end{definition}

\begin{theorem}\label{mass1}\textup{(Mass formula of Ekedahl 
and Hashimoto/Ibukiyama)}
\[
m_g([\ell],p)_{(0)}=M_g(p):=\frac{(-1)^{g(g+1)/2}}{2^g}\left\{
\prod_{k=1}^g \zeta(1-2k)\right\}\cdot\prod_{k=1}^g\{p^k+(-1)^k\}.
\]
\end{theorem}
\begin{proof}
See \cite[p.~159]{Ek} and \cite[Prop.~9]{HI}, cf.
\cite[Thm.~3.1]{Y}.
\end{proof}

\begin{proposition}\label{mass2}
\[
m_g([\ell],p)_{(r_0,\ldots,r_{k-1},r_{k})}=m_g([\ell],p)_{(r_0,r_1,\ldots,r_{k-1})}
N(\ell)_{r_{k-1},r_{k}}
\]
where $N(\ell)_{r_{k-1},r_{k}}$ is as given in \eqref{fly} and
Proposition \textup{\ref{mouse}}.
\end{proposition}
\begin{proof}
Let $\sigma$ be the $(k-1)$-simplex of type $(r_0,r_1,\ldots,r_{k-1})$
corresponding to the
sequence $$\A_0\xrightarrow{f_1}\A_1\xrightarrow{f_2}
\cdots\xrightarrow{f_{k-1}}\A_{k-1}.$$ Let $G(\sigma)$ be the automorphism
group of this sequence, so $w(\sigma)=\#G(\sigma)$.

By Definition \eqref{fly}, there are $N(\ell)_{r_{k-1},r_{k}}$ maps
$f_k\in\Iso_\ell(\A_{k-1})_{r_{k}}$ from $\A_{k-1}$ to some $\A_{k}$ of type
$r_k$.  Two such $f_k$'s extend $\sigma$ to the same $k$-simplex
$\sigma[f_k]$ if and only if there exists an 
element $g\in G(\sigma)$ sending one
to the other.  Let $\Orb_{G(\sigma)}(f_k)$ be the orbit under
$G(\sigma)$ of $f_k$. We  have $$\frac{N(\ell)_{r_{k-1},r_{k}}}{w(\sigma)}=
\sum_{f_k\in\Iso_\ell(\A_{k-1})_{r_k}}\frac1{\#G(\sigma)}=
\sum_{\sigma[f_k]}\frac{\#\Orb_{G(\sigma)}(f_k)}{\#G(\sigma)}=
\sum_{\sigma[f_k]}\frac1{\#G(\sigma[f_k])}=
\sum_{\sigma[f_k]}\frac1{w(\sigma[f_k])}.$$ Summing the above over all
$(k-1)$-simplexes of type $(r_0,r_1,\ldots,r_{k-1})$ gives the result.
\end{proof}

\begin{proposition}\label{mass3}
\[
m_g([\ell],p)_{(r_0,r_1,\ldots,r_k)}=m_g([\ell],p)_{(\hat{r}_k,\ldots,\hat{r}_1,
\hat{r}_0)}.
\]
\end{proposition}
\begin{proof}
The dual map $\iota$ send $k$-simplexes of type
$\bfr=(r_0,r_1,\ldots,r_k)$ to $k$-simplexes of type
$\hat{\bfr}=(\hat{r}_k,\ldots,\hat{r}_1,\hat{r}_0)$ and preserves weights.
\end{proof}

\begin{corollary}
\[
m_g([\ell],p)_{(r)}=M_g(p)\frac{N(\ell)_{g,\hat{r}}}{N(\ell)_{r,0}}.
\]
\end{corollary}
\begin{proof}
Using Propositions \ref{mass2} and \ref{mass3} we
compute 
\begin{align*}
m_g([\ell],p)_{(0)}N(\ell)_{g,\hat{r}}&=
m_g([\ell],p)_{(g)}N(\ell)_{g,\hat{r}}= m_g([\ell],p)_{(g,\hat{r})}\\
&=m_g([\ell],p)_{(r,0)}= m_g([\ell],p)_{(r)}N(\ell)_{r,0}.
\end{align*}
\end{proof}

Combining the above results we get:

\begin{theorem}
\label{gutter}
\[
m_g([\ell],p)_{(r_0,r_1,\ldots,r_{k})}=
M_g(p)\frac{N(\ell)_{g,\hat{r}_0}}{N(\ell)_{r_0,0}}\prod_{i=1}^kN(\ell)_{r_{i-1},r_{i}}.
\]
\end{theorem}

\section{The Brandt complex of a definite rational quaternion
algebra}
\label{brandtc}

There are several ways to define the Brandt complex of a definite
rational quaternion algebra.  Here we present the one we used in our
computations.

Let $\HH$ be a definite rational quaternion algebra and $\cO_{\HH}\subset \HH$ a
maximal order. Recall that a {\sf positive definite hermitian form}
over $\cO_{\HH}$ is a projective left $\cO_{\HH}$-module $M$ of rank $g$ with a
$\Z$-linear pairing $(\,\,\, ,\,\,)$ into $\cO_{\HH}$ satisfying
\begin{enumerate}
\item $(\alpha x,y)=\alpha(x,y)$,
\item $(x,y)=\overline{(y,x)}$,
\item $(x,x)\ge0$ with equality only if $x=0$
\end{enumerate}
for all $\alpha\in\cO_{\HH}$ and $x,y\in M$.

Given such a form recall that its {\sf dual} $M^*$ is defined as
$\{x\in M\otimes\HH \mid (M,x)\subset\cO_{\HH}\}$.

\begin{definition}
For a prime $\ell$ unramified in $\HH$, a positive definite hermitian
form $M$ is {\sf $\ell$-bounded}  if
\begin{equation}
\label{ls}
\ell M^*\subset M.
\end{equation}
\end{definition}

Observe that by comparing the reduced discriminant of both sides of
equation \eqref{ls} we conclude that the reduced discriminant
$\disc(M)=\ell^n$ for some integer $0\le n\le g$.  Given 
$M,\,(\,\,\,,\,\,)$ an
$\ell$-bounded positive definite hermitian form, we define the {\sf
  scaled dual} form (denoted $\widehat{M}$) to be $M^*,\,
\ell(\,\,\,,\,\,)$.

\begin{remark1}\label{sf}
{\rm

There are two ways to think about $\ell$-bounded forms.
\begin{enumerate}[\upshape (a)]
\item\label{esl} Recall that all positive definite $g$-dimensional hermitian
  forms over $\HH$ are isomorphic. Thus we can fix the form 
$(\,\,\,,\,\, )$
  on $\HH^g\cong M\otimes\HH$ and vary $M\subset \HH^g$.  In fact we can
  simultaneously fix $M\otimes\Z[1/\ell]=\cO_{\HH}[1/\ell]^g$.
\item Also recall that for $g>1$ all projective rank $g$ $\cO_{\HH}$-modules
  are isomorphic. Thus for $g>1$ we can alternatively fix
  $M\cong\cO_{\HH}^g$ and vary the $\ell$-bounded hermitian form 
$(\,\,\, ,\,\,)$, given
  by $(x,y)=y^\dagger Hx$ for a hermitian matrix $H$.
\end{enumerate}

}
\end{remark1}

We will now define the {\sf enhanced Brandt complex} 
$\wbrc_{g}(\ell,\cO_{\HH})$ of
  $\cO_{\HH}$ with respect to the prime $\ell$ of dimension $g$.  The
vertices of $\wbrc_{g}(\ell,\cO_{\HH})$ correspond to the isomorphism classes of
$\ell$-bounded positive definite hermitian forms $M$ over $\cO_{\HH}$ of
rank $g$.  If the reduced discriminant $\disc(M)=\ell^n$ we will say
the corresponding vertex has {\sf type $n$}.

The $k$-faces of the enhanced Brandt complex correspond to chains of
$\ell$-bounded positive definite hermitian forms $M_0\subsetneq
M_1\subsetneq \ldots\subsetneq M_{k}$ with the same pairing $(\,\,\,,\,\,\,)$.  
Such a
face will be said to be of {\sf type $\mathbf{n}\colonequals
(n_0,n_1,\ldots,n_{k})$} if
$\disc(M_i)=\ell^{n_i}$.

The {\sf little Brandt complex}  $\brc_{g}(\ell,\cO_{\HH})$ is obtained from the
enhanced Brandt complex by quotienting by the involution corresponding
to taking scaled duals.

As preparation for the proof of the theorem below, the
{\sf adjugate} (or {\sf classical adjoint})
 $\Adjj(M)$ of an invertible $n\times n$-matrix $M$ is the $n\times n$
matrix with $M\Adjj(M)=\det M \Id_{n\times n}$.

\begin{theorem}\label{Brandt T} The enhanced Brandt complex $\wbrc_{g}(\ell,\cO_{\HH})$ is
isomorphic to
\begin{equation}\label{Brandt U}
\rU_g(\cO_{\HH}[1/\ell])\bs\cB_g ,
\end{equation}
and the little Brandt complex $\brc_{\!g}(\ell,\cO_{\HH})$ is isomorphic to
\begin{equation}\label{Brandt GU}
\GU_g(\cO_{\HH}[1/\ell])\bs\cB_g .
\end{equation}
\end{theorem}
\begin{proof}
Since $\HH$ is unramified at $\ell$, we can fix an isomorphism
$\HH\otimes\Q_\ell\cong\Mat_{2\times2}(\Q_\ell)$ such that
$\cO_\HH\otimes\Z_\ell$ corresponds to $\Mat_{2\times2}(\Z_\ell)$.
Under this isomorphism conjugating a quaternion corresponds to taking
the adjugate matrix.  For an invertible $2\times2$ matrix $m$ the adjugate
is given by $\Adj(m)=sm^Ts^{-1}$, where $s=\begin{pmatrix} 0 & 1\\-1 &
0\end{pmatrix}$ is the matrix of a symplectic form.  Interpreting
$\HH^g$ as row vectors, we get an isomorphism
$\HH^g\otimes\Q_\ell\cong\Mat_{2\times2g}(\Q_\ell)\cong V\oplus V$
where $V=\Q_\ell^{2g}$ as a row vector space.  Since
$\HH\otimes\Q_\ell\cong\Mat_{2\times2}(\Z_\ell)$ acts on
$\Mat_{2\times2g}(\Q_\ell)$ on the left, any
$\Mat_{2\times2}(\Z_\ell)$-lattice
$W'\subset\Mat_{2\times2g}(\Q_\ell)$ is of the form $W\oplus W$ for
$W\subset V$ a $\Z_\ell$-lattice.  Also the hermitian form $(x,y)$ on
$\Mat_{2\times2g}(\Q_\ell)$ is given by $xs^{\oplus^g}y^Ts^{-1}$,
where $s^{\oplus^g}$ is the block diagonal matrix consisting of $g$
copies of $s$ which gives $V$ the structure of a $2g$-dimensional
symplectic space over $\Q_\ell$ as used in Section \ref{quotient} to
define the building.

Define a map $\upsilon:\cB_g\to\wbrc_{g}(\ell,\cO_{\HH})$ that maps
the vertex corresponding to the extra special
$((1)^r,(0)^{g-r})$-lattice $W\subset V$ to the hermitian form given
by $M=(W\oplus W)\cap\cO_{\HH}[1/\ell]^g$.  Note that $M$ will be
$\ell$-bounded of type $r$ since $W$ was extra special.  Extend
$\upsilon$ to higher faces in the natural manner, namely by mapping
the face corresponding to the chain $W_0\subset
W_1\subset\ldots\subset W_k$ to $(W_0\oplus
W_0)\cap\cO_{\HH}[1/\ell]^g\subset (W_1\oplus
W_1)\cap\cO_{\HH}[1/\ell]^g\subset\ldots\subset (W_k\oplus
W_k)\cap\cO_{\HH}[1/\ell]^g$.  By Remark \ref{sf}\eqref{esl} every
$\ell$-bounded $M$ can be embedded into $\cB_g$ with
$M\otimes\Z[1/\ell]=\cO_{\HH}[1/\ell]^g$ so comes from the extra
special lattice $W$ with $W\oplus W=M\otimes\Z_\ell\subset V\oplus V$.
Hence $\upsilon$ is surjective.

Now suppose we have two modules $M,M'\subset\HH^g$ with
$M\otimes\Z[1/\ell]=M'\otimes\Z[1/\ell]=\cO_{\HH}[1/\ell]^g$.  Any
isomorphism between them extends to an automorphism of
$\cO_{\HH}[1/\ell]^g$, i.e., an element of $\rU_g(\cO_{\HH}[1/\ell])$.
The same logic applies to chains.  Thus the image of two faces under
$\upsilon$ correspond to the same element of the enhanced Brandt
complex if and only if they're related by an element of
$\rU_g(\cO_{\HH}[1/\ell])$.  This proves \eqref{Brandt U}.

Now pick a module $M\subset\HH^g$ with
$M\otimes\Z[1/\ell]=\cO_{\HH}[1/\ell]^g$ and consider an embedding of
the scaled dual $\hat{M}\cong M'\subset\HH^g$ also with
$M'\otimes\Z[1/\ell]=\cO_{\HH}[1/\ell]^g$.  This gives us a scaling
from $M^*$ to $M'$ which extends to an element of
$\GU_g(\cO_{\HH}[1/\ell])$.  Since $M\otimes\Z_\ell$ and
$M^*\otimes\Z_\ell$ give elements of the same generalized homothety,
we see that the preimages of $M$ and $\hat{M}$ under $\upsilon$ are
related by an element of $\GU_g(\cO_{\HH}[1/\ell])$.  Extending the
same logic to higher faces and using that both sides are quotients of
order $2$ gives \eqref{Brandt GU}.
\end{proof}

\begin{corollary}
The enhanced and little Brandt complexes are independent of the choice
of maximal order $\cO_{\HH}$ and depend only on the algebra $\HH$.
\end{corollary}
\begin{proof}
This follows from the fact that for $\ell$ unramified in $\HH$, the
algebra $\cO_{\HH}[1/\ell]$ is independent of the maximal order
$\cO_{\HH}$ and depends only on $\HH$.
\end{proof}

\begin{corollary}
If $A_0=E^g$ is a principally polarized superspecial abelian variety of
dimension $g$ with $\cO=\End(E)$, then the enhanced
\textup{(}respectively, little\textup{)} $\ell$-complex of 
$\A_0$ is isomorphic to the enhanced
Brandt complex $\wbrc_{g}(\ell,\cO)$ \textup{(}respectively, the little
 Brandt complex
$\brc_{g}(\ell,\cO)$.\textup{)}.
\end{corollary}
\begin{proof}
Combine Theorem \ref{Brandt T} with Theorem \ref{pav T}.
\end{proof}

\subsection{Computations of Brandt complexes}
\label{rounder}

 Let $N$ be a square-free positive integer with an odd number of prime
factors.  Let $\cO_\HH$ be a maximal order in the definite rational
quaternion algebra $\HH=\HH(N)$ of reduced discriminant $\Disc(\HH)=N$.  We computed
the Brandt complexes 
$\wbrc_{g}(\ell,\cO_\HH)$  and $\brc_{g}(\ell,\cO_{\HH})$
for the values of $N$ with $(N,\ell)=1$
in Table \ref{oak}.
\begin{center}
\begin{table}[h]
\begin{tabular}{c|c|c}
$g$ & $\ell$ & $N=\Disc(\HH)$\\
\hline
$2$ & $2$ & $N\le347$\\
  & $3$ & $N\le227$\\
  & $5$ & $N\le157$\\
  & $7$ & $N\le107$\\
\hline
$3$ & $2$ & $N\le23$\\
  & $3$ & $N\le13$
\end{tabular}\\[.25in]
\caption{The range of our computations 
of $\wbrc_{g}(\ell,\cO_\HH)$  and $\brc_{g}(\ell,\cO_{\HH})$}
\label{oak}
\end{table}
\end{center}

\section{\texorpdfstring{$\wco_{2}(2,7)$}{tilde co\unichar{"5F}2(2,7)} and \texorpdfstring{$\co_{2}(2,7)$}{co\unichar{"5F}2(2,7)}}
\label{dishes}

In this section we examine the isogeny complexes
$\wco_{2}(2,7)$ and $\co_{2}(2,7)$ in detail, giving their
faces and half-faces in each dimension.  In \cite{jz2} we build on
this, using these examples to illustrate how we compute the cohomology 
of isogeny complexes.  We  give the standard cochain complexes with their differentials
for $\wco_2(2,7)$ and $\co_2(2,7)$ and then compute the cohomology
$H^\ast(\wco_{2}(2,7),\Z)$ and $H^\ast(\co_{2}(2.7),\Z)$.

\subsection{\texorpdfstring{\except{toc}{\boldmath{$\wco_2(2,7)=\wbrc_2(2,\OO_{\HH(7)})$}}\for{toc}{$\wco_2(2,7)=\wbrc_2(2,\OO_{\HH(7)})$}}{tilde co\unichar{"5F}2(2,7)=tilde brc\unichar{"5F}2(2,O\unichar{"210D}(7))}}
\label{radishes}

The cell complex $\wco_2(2,7)=\wbrc_2(2,\OO_{\HH(7)})$ has $8$ $0$-cells, 
$23$ $1$-cells, and $16$ $2$-cells.  They are divided
among the possible types as shown in Table \ref{snap}.

\begin{center}
\begin{table}[h]
\begin{tabular}{c||c|c|c|c|c|c|c}
 \text{Type} & $(0)$ & $(1)$ & 
 $(2)$ & $(1,0)$ & 
$(2,0)$ & $(2,1)$ & $(2,1,0)$\\
\hline
\text{Number of cells} &  $2$ & $4$ & $2$ & $7$ &  $9 $ & $7$ & $16$
\end{tabular}\\[.25in]
\caption{Number of cells in $\wco_2(2,7)$ by type}
\label{snap}
\end{table}
\end{center}

The weights and masses of the cells of the various types are given
in Table \ref{snap2}.
\begin{center}
\begin{table}[h]
\begin{tabular}{c|c|c}
 \text{Type $t$} & \text{Weights} & \text{Mass $m_2(2,7)_t$}\\
\hline\hline
$(0)$ & $32,48$& $1/32+1/48=5/96$\\
$(1)$ & $16,16,8,96$ & $25/96$\\
$(2)$ & $32,48$&$ 5/96$\\
$(1,0)$ & $8, 8, 32, 16, 8, 8, 12, 12, 48$ & $25/32$\\
$(2,0)$ & $8, 16, 8, 8, 32, 4, 16$ & $25/32$\\
$(2,1)$ & $8, 16, 8, 16, 8, 4, 32$ & $25/32$\\
$(2,1,0)$ & $ 4, 8, 8, 16, 8, 8, 8, 8, 4, 32, 16, 4, 4, 4, 8, 16$ & $75/32$
\end{tabular}\\[.25in]
\caption{Weights of the cells in $\wco_2(2,7)$ by type}
\label{snap2}
\end{table}
\end{center}
From Theorem \ref{mass1} we have
\[
m_2(2,7)_{(0)}=\frac{(-1)^{3}}{4}\left\{\zeta(-1)\zeta(-3)\right\}\cdot
(7-1)(7^2+1)
= \frac{-1}{4}\left\{\frac{-1}{12}\cdot\frac{1}{120}\right\}\cdot(300)
= \frac{5}{96}\, ,
\]
agreeing with the computation of $m_2(2,7)_{(0)}$ in Table \ref{snap2}.
The other $m_2(2,7)_t$ can now be computed from Theorem
\ref{gutter} using Table \ref{dodge2} for the $N(2)_{r,s}$ -- reassuringly
they all agree with Table \ref{snap2}.

\subsection{\texorpdfstring{\except{toc}{\boldmath{$\co_2(2,7)=\brc_2(2,\OO_{\HH(7)})$}}\for{toc}{$\co_2(2,7)=\brc_2(2,\OO_{\HH(7)})$}}{co\unichar{"5F}2(2,7)=brc\unichar{"5F}2(2,O\unichar{"210D}(7))}}

The cell complex $\co_2(2,7)=\brc_2(2,\OO_{\HH(7)})$ has $6$ $0$-cells, 
$14$ $1$-cells, and $12$ $2$-cells.  They are divided
among the possible types as in Table \ref{snap3} .

\begin{center}
\begin{table}[h]
\begin{tabular}{c||r|c|r|c|c|c|c}
 \text{Type} & $\overline{(0)}$  & $\overline{(1)}$ & 
 $\overline{(1,0)}$ & $\overline{(2,0)}$: & 
$\overline{(2,0)}$: & $\overline{(2,1,0)}$: & $\overline{(2,1,0)}$:\\
& $=\overline{(2)}$  & & $=\overline{(2,1)}$ & regular edges & half-edges 
& regular facets & half-facets\\
\hline
\text{\# cells} &  $2\,\,\,\,\,$ & $4$ & $7$\,\,\,\,\,\,\,\,   & $2$ &  $5 $ & $4$ & $8$
\end{tabular}\\[.25in]
\caption{Number of cells in $\co_2(2,7)$ by type}
\label{snap3}
\end{table}
\end{center}

\bibliographystyle{plain}
\bibliography{ICg5}%{}%../local,outside

\end{document}